\documentclass{article}
\usepackage[english]{babel}
\usepackage{amsfonts, amsmath, amsthm, amssymb,amscd,indentfirst}
\usepackage{marginnote}
\usepackage{bbding}
\usepackage{xcolor}
\usepackage[scr=boondox]{mathalfa}
\usepackage{amsmath,amssymb,latexsym,indentfirst}
\usepackage{times}
\usepackage{palatino}
\usepackage{MnSymbol}
\usepackage[top=1.5in, bottom=1.5in, left=1.1in, right=1.1in]{geometry}

\newtheorem{theorem}{Theorem}[section]

\newtheorem{proposition}[theorem]{Proposition}

\newtheorem{definition}[theorem]{Definition}

\newtheorem{conjecture}[theorem]{Conjecture}

\newtheorem{example}[theorem]{Example}
\newtheorem{remark}[theorem]{Remark}

\def\a{{\alpha}}

\def\a{{\alpha}}

\def\a{{\alpha}}

\begin{document}

\title{A mass-type invariant for smooth metric measure spaces and its relation with the fractional Yamabe problem}

\author{\textsc{S\'ergio Almaraz}\footnote{Partially supported by grant 201.049/2022, FAPERJ/Brazil, and grant 304894/2025-7, CNPq/Brazil.},  \textsc{Levi Lopes de Lima},  \textsc{and Shaodong Wang}\footnote{Partially supported by NSFC 12001364, the Fundamental Research Funds for the Central Universities, No.30924010839 and the Stability Support Program for Basic Research (WZJC202601002).}}

\maketitle

\begin{abstract}
On asymptotically Schwarzschild smooth metric measure spaces modeled on the Euclidean half-space, we define a new mass-type invariant analogous to the ADM mass of an asymptotically flat manifold. We prove that this invariant is independent of the choice of admissible asymptotic coordinates. We further show that it fits into a very general framework for mass-type invariants and we also isolate static equations for smooth metric measure spaces under natural boundary conditions. For smooth metric measure spaces obtained through a generalized stereographic projection using the Green's function associated with the fractional Yamabe problem, we identify the mass with the constant term in the asymptotic expansion of the Green's function. This identification gives a geometric interpretation of the constant appearing in the fractional Yamabe problem and recasts its conjectured positivity as a positive-mass-type question.
\end{abstract}

%\tableofcontents
%\newpage

\section{Introduction}\label{sec:intro}
Arnowitt, Deser and Misner discovered in~\cite{arnowitt-deser-misner} a key quantity in the general relativity theory which stands for the total mass of an isolated gravitational system. Mathematically, this system is represented by a Riemannian manifold asymptotically modelled on the Euclidean space $\mathbb R^N$. That mass quantity, named the "ADM mass" after the authors, was later proved by Bartnik in \cite{bartnik} and by Chrusciel~\cite{chrusciel} to be independent of the chosen asymptotic coordinates.

The non-negativity of the ADM mass in case the physically isolated system is time-symmetric (when embedded in a Lorentzian space-time) under appropriate energy conditions is known as the "Riemannian positive mass theorem" and was proved by Schoen and Yau in \cite{schoen2, schoen-yau1} in dimensions $3\leq N\leq 7$, and by Witten in \cite{witten} in any dimension $N\geq 3$ assuming the manifold is spin. The results in \cite{schoen2, schoen-yau1} played a crucial role in the proof of the remaining cases of the Yamabe problem by Schoen in \cite{schoen}. There has since been extensive work on mass-type geometric quantities under various assumptions. Extensions to high dimensions of the results in \cite{schoen2, schoen-yau1} were obtained by Schoen and Yau in \cite{schoen-yau3} and more recently by Bi, Hao, He, Shi and Zhu in \cite{bi-hao-he-shi-zhu} and by Brendle and Wang in \cite{brendle-wang}.
We refer the reader to \cite{lee} for an introduction on the subject. 

In the case of manifolds with non-compact boundary, a mass-type geometric invariant was studied in \cite{almaraz-barbosa-lima} for Riemannian manifolds asymptotically modelled on the Euclidean half-space
$$
\mathbb R^{N}_+=\{x=(x_1,...,x_{N})\in \mathbb R^{N}\:|\:x_N\geq 0\},
$$
where a positive mass theorem was also established. This mass invariant is strongly motivated by problems related to the Yamabe problem on manifolds with boundary proposed by Escobar in \cite{escobar2}.

%The purpose of this paper is to study a new mass-type geometric invariant quantity in smooth metric measure spaces asymptotically modelled on $\mathbb R_+^{N}$ and relate it with the fractional Yamabe problem \cite{gonzalez-qing} which is an extension of the classical Yamabe problem. While the Escobar-Yamabe problem in \cite{escobar2} can also be viewed as a particular case of the fractional Yamabe problem, we show that the mass in \cite{almaraz-barbosa-lima} appears as a particular case of the one proposed here for smooth metric measure spaces. The main motivation for studying this invariant was to give a geometrical interpretation for a key term in the Green's function expansion \cite[Conjecture 1.6]{kim-musso-wei} which appears in connection with the fractional Yamabe problem and was studied in \cite{mayer-ndiaye1, mayer-ndiaye2}.

The purpose of this paper is to study a new mass-type geometric invariant for smooth metric measure spaces asymptotically modelled on $\mathbb R_+^{N}$ and to relate it to the fractional Yamabe problem \cite{gonzalez-qing}, which is an extension of the classical Yamabe problem. More generally, we introduce a notion of static smooth metric measure space with boundary and associate to each static potential of an asymptotic model a total charge. Our mass arises as the charge associated with the constant static potential of the Euclidean half-space. While the Escobar-Yamabe problem in \cite{escobar2} can also be viewed as a particular case of the fractional Yamabe problem, we show that the mass in \cite{almaraz-barbosa-lima} appears as a particular case of the mass proposed here for smooth metric measure spaces. Our main motivation for studying this invariant is to give a geometric interpretation of a key term in the Green’s function expansion \cite[Conjecture 1.6]{kim-musso-wei} which appears in connection with the fractional Yamabe problem and was studied in \cite{mayer-ndiaye1, mayer-ndiaye2}.

Smooth metric measure spaces (SMMS for short) are $(n+1)$-dimensional Riemannian manifolds $(M,g)$ equipped with a measure $v^m dvol_g$ which is multiple of the metric's volume element $dvol_g$ by a non-negative smooth function $v^m$, with $v^{-1}(0)$ being the boundary of $M$ and $m\in \mathbb{R}\cup \{\pm \infty\}$.
The interest in such objects has increased in the last decades and one major question is how geometric invariants are adapted to this setting. 
The classical example is the Bakry-Emery Ricci tensor introduced in \cite{bakry-emery}. 
Another important geometric invariant of a SMMS is its weighted scalar curvature $R_{(g,v,m)}$ (see formula \eqref{eq:R} below) which extends the scalar curvature $R_g$ of a Riemannian manifold. In this paper, we provide the formula of what we call the "modified weighted mean curvature" $\mathcal H_{(g,v,m)}$  in \eqref{eq:H}. As the name suggests, this is obtained by modifying the weighted mean curvature (see \cite{gromov}) which itself extends the mean curvature $H_g$ of a smooth hypersurface on a Riemannian manifold. The quantity $\mathcal H_{(g,v,m)}$  appears naturally in our setting as it couples perfectly with the boundary contributions from $R_{(g,v,m)}$ which we will make clear in Section~\ref{sec:interpr:mass}; see also Remark \ref{rmk:inter:H} for an interpretation of that geometric invariant.

%Also relevant to our context is the modified weighted mean curvature $\mathcal H_{(g,v,m)}$  (see \eqref{eq:H}) which extends the mean curvature $H_g$ of a smooth hypersurface on a Riemannian manifold. %Both $R_{(g,v,m)}$ and $\mathcal H_{(g,v,m)}$ are used in Section \ref{sec:interpr:mass} for motivating the mass expression \eqref{eq:mass:0}.

We assume $n\geq 2$ and work with metrics expanded as 
\begin{equation}\label{exp:a:s}
g_{ij}(x)=\left(1+a|x|^{1-n-m}\right)\delta_{ij}+O'(|x|^{\epsilon-n-m}),\qquad i,j=1,...,n+1,
\end{equation}
where $\delta$ is the Euclidean metric, and weight functions $v$ in $v^mdvol_g$ expanded as 
\begin{equation}\label{exp:a:s:v}
v(x)=\left(1+\displaystyle\frac{a}{2}|x|^{1-n-m}\right)x_{n+1}+x_{n+1}O'(|x|^{\epsilon-n-m})
\end{equation}
in the asymptotic coordinates $x=(x_1,...,x_{n+1})\in\mathbb R^{n+1}_+$, for some constants $a$ and $0<\epsilon<1$. 
Here, and in the rest of the paper, $O'(r^{\a})$ stands for a term whose norm is bounded by $Cr^{\a}$ and the norm of its derivative is bounded by $Cr^{\a-1}$ for some $C>0$.
Precisely, we work with what we call an asymptotically Schwarzschild SMMS $(M,g,v^m dvol_g)$ as in Definition~\ref{def:as:flat} below.
This is a natural extension of the classical asymptotically Schwarzschild metrics which play a major role in the general relativity theory; see also Remark \ref{rmk:AS} below. We define the mass of  $(M,g,v^m dvol_g)$ by the expression
	\begin{align}\label{eq:mass:0}
		\mathfrak m_{(g,v,m)}=\lim_{\rho\to\infty}\lim\limits_{s\to 0^+}\Big \{
		&\int_{|x|=\rho, x_{n+1}>s}\sum_{i=1}^{n+1}\mathbb U_i(g-\delta, v-x_{n+1}) \frac{x_i}{|x|}x_{n+1}^md\sigma
		\\
		&+\int_{|x|=\rho, x_{n+1}= s}\sum_{\alpha=1}^{n}\overline {\mathbb U}_\alpha(g-\delta, v-x_{n+1}) \frac{x_\alpha}{|x|}x_{n+1}^m d\overline \sigma
		\Big \},\notag
	\end{align}
	where 
	$$
	\mathbb U(e,w)=\text{div}_\delta e-d\text{tr}_\delta e+mx_{n+1}^{-1}\nabla_\delta x_{n+1}\righthalfcup e-2mx_{n+1}^{-1} dw
	$$ 
	and $\overline {\mathbb U}(e,w)=\nabla_\delta x_{n+1}\righthalfcup e$ are the {\rm{charge integrand}} one-forms, and $d\sigma$ and $d\overline\sigma$ are Euclidean area elements. 
Here, $e$ is a symmetric two-tensor on $M$ and $w$ is a smooth function on $M$. Similar to other geometric objects in SMMS, our definition has a nice warped-product interpretation as developed in Remark \ref{rmk:warped} below. 

Our first main result is that the limit in \eqref{eq:mass:0} is finite, independent of the choice of asymptotic coordinates for the SMMS  $(M,g,v^m dvol_g)$, and recovers the constant $a$ in expansions \eqref{exp:a:s} and \eqref{exp:a:s:v}.  
This reinforces the analogy with  asymptotically Schwarzschild metrics, where the ADM mass is a geometric invariant of the asymptotic structure which recovers a similar constant in its asymptotic metric expansion.
\begin{theorem}\label{thm:main}
	Suppose $0\leq m<1$, or $-1<m<0$ with $O'(|x|^{\epsilon-n-m})=O'(|x|^{\epsilon-n-m})\delta_{ij}+x_{n+1}O'(|x|^{\epsilon-n-m-1})$ in \eqref{exp:a:s}.
	The limit defining $\mathfrak m_{(g,v,m)}$ in \eqref{eq:mass:0} exists (and is finite) and 
	$$
	\mathfrak m_{(g,v,m)}=(m+n)(m+n-1)a \sigma_{n,m},
	$$
	where 
	$$
	\sigma_{n,m}=\int_{|x|=1, x_{n+1}> 0}x_{n+1}^m d\sigma.
	$$
	This limit is independent of the chosen asymptotic coordinates.
\end{theorem}

Theorem \ref{thm:main}, which is proved in Sections \ref{sec:mass:def} and \ref{sec:geom:inv}, characterizes $	\mathfrak m_{(g,v,m)}$ as a geometric invariant of the asymptotically Schwarzschild SMMS $(M,g,v^mdvol_g)$. The assumption $O'(|x|^{\epsilon-n-m})=O'(|x|^{\epsilon-n-m})\delta_{ij}+x_{n+1}O'(|x|^{\epsilon-n-m-1})$ in the expansion of $g$ in \eqref{exp:a:s} for the case $-1<m<0$ in the theorem is justified by the model case when the asymptotically Schwarzschild SMMS comes from the stereographic projection of the Green's function, which is the content of Theorem~\ref{thm:second} below.
Other mass-type invariants for SMMS without boundary were studied in \cite{baldauf-ozuch, chu-zhu, law-lopez-santiago}; see Remark \ref{rmk:baldauf-ozuch}.

The general construction leading to the mass expression \eqref{eq:mass:0} has interest on its own and is developed in Section~\ref{sec:interpr:mass}.
The resulting system extends the concept of static manifolds with boundary (see \cite[Definition 1.2]{almaraz-lima2}) coming from equations independently introduced in \cite{cruz-vitorio} and \cite{almaraz-lima1}.
It follows a general procedure for obtaining a mass-type total charge seen as a linear functional on the space of static potentials of a given static underlying manifold. This construction was put in a  very general framework by Michel in ~\cite{michel}. In the present paper, through an explicit and delicate linearization of the weighted scalar curvature and the corresponding modified weighted mean curvature, we show how our expression fits into this general framework and gives rise to generalized static equations \eqref{sys:static} for smooth metric measure spaces with boundary conditions \eqref{sys:static:bdry}. We believe our constructions open a promising research line, namely to explore not only the static conditions \eqref{sys:static} and \eqref{sys:static:bdry}, but also the linear functional \eqref{limit:2} on the space of static potentials.

%\begin{remark}
%The charge density $\mathbb U$ may have a singular behavior on the boundary due to negative powers of $x_{n+1}$. The hypothesis $m\geq 0$ is necessary to ensure the convergence of the integral involving $\mathbb U$ in  \eqref{eq:mass:0} (see the proof of Proposition \ref{lemma:mass}). On the other hand, ${\overline {\mathbb U}}$ has no negative powers of $x_{n+1}$ in its definition so that $m\geq 0$ guaranties that the integral involving ${\overline {\mathbb U}}$ in \eqref{eq:mass:0} makes sense. In particular, that integral vanishes when $m>0$.
%In view of the weight $x_{n+1}^m$ in the integrals in \eqref{eq:mass:0}, the hypothesis $m\geq 0$ is necessary for essentially two reasons. Firstly, it is used to ensure the convergence of the integrals near to $x_{n+1}=0$ on hemispheres. Secondly, it ensures that the boundary terms indeed make sense.  However, formally all the corresponding computations still work for $m<0$.  Also observe that the integral involving $\overline {\mathbb U}$ in \eqref{eq:mass:0} vanishes when $m>0$.
%\end{remark}

The relevance of this mass quantity comes from its relation with the $\gamma$-fractional Yamabe problem and a relevant Green's function. 
This relies on the fractional GJMS operator, also known as the conformal fractional Laplacian operator, which is a non-local operator of fractional order $2\gamma$ defined on the boundary of a conformally compact Einstein manifold by using scattering theory. 
Kim, Musso and Wei  proved in \cite{kim-musso-wei} that the $\gamma$-fractional Yamabe problem is solvable by a minimization process when either $n=2$ or $n\geq 3$ and the boundary manifold is locally conformally flat and the Green's function of the conformal fractional Laplacian is expanded as
\begin{equation}\label{eq:green:0}
G_p(q)=d(p,q)^{2\gamma-n}+A+O'(d(q,p)^{\min\{1,2\gamma\}}),
\end{equation}
with $A$ being a positive constant and $d(\cdot,\cdot)$ the distance function on the manifold (see Conjecture \ref{conj:green}). 
%\vspace{0.2cm}
%\noindent{\bf{Notation.}}
%Here, and in the rest of the paper, $O(r)$ stands for a term whose norm is bounded by $Cr$ for some $C>0$.
%\vspace{0.2cm}
%Wang and Zhou~\cite{wang-zhou} recently announced the proof of Conjecture~\ref{conj:green}, namely, that expression ~\eqref{eq:green:0} holds with a positive constant $A$. 
The expansion \eqref{eq:green:0}, regardless the sign of $A$, was obtained by Mayer and Ndiaye in \cite{mayer-ndiaye1} and an alternate solution for the $\gamma$-fractional Yamabe problem was obtained by the same authors in \cite{mayer-ndiaye2} by means of Bahri-Coron methods.

In our paper, we provide a geometric interpretation for the constant $A$ in \eqref{eq:green:0} by identifying it with the mass $\mathfrak m_{(g,v,m)}$ in ~\eqref{eq:mass:0} of a corresponding asymptotically Schwarzschild smooth metric measure space. This is the content of the next result.
\begin{theorem}\label{thm:second}
Let $0<\gamma<1$, $\gamma\neq 1/2$, and $m=1-2\gamma$, and suppose that either $n=2$ or the conformal infinity is locally conformally flat. Moreover, assume there exists a positive Green's function as in Proposition \ref{propo:green}. Then the stereographic projection of the compactification of a conformally compact Einstein manifold is an asymptotically Schwarzschild SMMS $(M,g,v^m dvol_g)$ with $O'(|x|^{\epsilon-n-m})=O'(|x|^{\epsilon-n-m})\delta_{ij}+x_{n+1}O'(|x|^{\epsilon-n-m-1})$ in \eqref{exp:a:s}. Its mass is a positive multiple of the constant $A$ in equation \eqref{eq:green:0}.
\end{theorem}
Stereographic projections here (see Definition \ref{def:ster:proj} below) are extensions of similar constructions used by Schoen in the solution of the Yamabe problem in \cite{schoen}.  
They make use of the Green's function for the fractional conformal Laplacian.
The precise statement of Theorem \ref{thm:second} is Theorem \ref{propo:as:flat} in Section~\ref{sec:mass:def}.

Supported by Theorem \ref{propo:as:flat} we conjecture that dominant energy conditions would imply non-negativity of the mass together with a rigidity statement as follows:

\begin{conjecture}\label{our:conj}
Given an asymptotically Schwarzschild SMMS $(M^{n+1},g,v^mdvol_g)$, suppose $-1<m<1$ and $dv\neq 0$ on $\Sigma$.
Assume also that $g$ and $v$ are smooth.
Let  $\mathcal H^{(s)}_{(g,v,m)}$ stand for the modified weighted mean curvature of $\Sigma_s=v^{-1}(s)$ for small $s>0$ 
and assume that $\int_{M}|R_{(g,v,m)}|v^mdvol_g<\infty$ and $\limsup_{s\to 0^+}\int_{\Sigma_s}|\mathcal H_{(g,v,m)}^{(s)}|s^m d\sigma_g<\infty$.
If
$$R_{(g,v,m)}\geq 0,\:\ \text{in}\:M\backslash\Sigma$$ and 
\begin{equation}\label{eq:lim:H}
\liminf_{s\to 0^+} s^m\mathcal H^{(s)}_{(g,v,m)}\geq 0\:\:\text{uniformly on compact sets of }\Sigma,
\end{equation}
then  $\mathfrak m_{(g,v,m)}\geq 0$. Moreover, when $m\neq 0$ the equality holds if and only if $(M,g,v^mdvol_g)$ is isometric to $(\mathbb R^{n+1}_+, \delta, x_{n+1}^mdx)$, i.e., there is a diffeomorphism $\phi:\mathbb R^{n+1}_+\to M$ such that $\phi^*g=\delta$ and $v\circ\phi(x)=x_{n+1}$.
%By analogy with the classical positive mass theorem, one could conjecture that the dominant energy conditions (non-negativity of the weighted scalar curvature and of the modified weighted mean curvature) together with their integrability would imply non-negativity of the mass. However, we find it hard to impose such a conjecture for essentially two reasons. The first one is that the integrability of those curvatures fail to ensure the well-definedness of our mass due to singular behaviors of our weighted curvatures. The second one is that it is not clear whether the modified weighted mean curvature is non-negative for the stereographic projection of a conformally compact Einstein manifold. Therefore, we do not have at the moment a strong evidence for the holding of a parallel of the classical positive mass theorem. This seems to be a very intriguing and difficult subject to be investigated in a future work. 
\end{conjecture}
\begin{remark}
Here and in other parts of the text, limits as \eqref{eq:lim:H} are understood after using the natural identification of the level sets of $v$ near $\Sigma$.
\end{remark}
The authors are currently investigating Conjecture \ref{our:conj} as an ongoing project.
The case $m=0$ is handled in \cite{almaraz-barbosa-lima} where a positive mass theorem is proved to be logically equivalent to the classical Riemannian positive mass theorem; see Remark 1.6 in that paper. The SMMS in this case becomes a classical Riemannian manifold $(M,g)$ with boundary $\Sigma$ and $R_{(g,v,m)}$ and $\mathcal H_{(g,v,m)}$ reduce to $R_g$ and $H_g$ respectively. The asymptotically Schwarzschild hypothesis \eqref{exp:a:s} ensures that $R_g$ is integrable on $M$ and $H_g$ is integrable on $\Sigma$. Under these hypotheses, the results of \cite{almaraz-barbosa-lima} ensure that the mass is non-negative when $R_g\geq 0$ and $H_g\geq 0$, and it is zero if and only if $(M,g)$ is isometric to $(\mathbb R^{n+1}_+,\delta)$.

A consequence of Conjecture \ref{our:conj} would be the following rigidity of $(\mathbb R^{n+1}_+, \delta, x_{n+1}^mdx)$ under compactly supported deformations preserving the dominant energy conditions:
\begin{conjecture}\label{our:corol}
Let $f:\mathbb R^n\to \mathbb R$ be a compactly supported smooth function and $\Sigma_f\subset \mathbb R^{n+1}$ be its graph. 
Suppose $g$ is a smooth Riemannian metric on $M_f=\{x\in\mathbb R^{n+1}\:|\:x_{n+1}\geq f(x_1,...,x_n)\}$ such that $g=\delta$ outside a compact set.
Let $v_f$ be any non-negative smooth function on $M_f$, coinciding with $x_{n+1}$ outside a compact set, such that $v_f^{-1}(0)=\Sigma_f$ and $dv_f\neq $ on $\Sigma_f$. Consider the SMMS $(M_f,g,v_f^mdvol_g)$, with $-1<m<0$ or $0<m<1$ and assume that 
$$\int_{M_f}|R_{(g,v_f,m)}|v^mdvol_g<\infty\qquad\text{and}\qquad\limsup_{s\to 0^+}\int_{(\Sigma_f)_s}|\mathcal H^{(s)}_{(g,v_f,m)}|s^md\sigma_g<\infty.$$
If 
$$R_{(g,v_f,m)}\geq 0,\:\ \text{in}\:M_f\backslash\Sigma_f
\qquad\text{and}\qquad\liminf_{s\to 0^+}s^m\mathcal H^{(s)}_{(g,v_f,m)}\geq 0\:\:\text{uniformly on }\Sigma_f,$$ 
then  $(M_f,g,v_f^mdvol_g)$ is isometric to  $(\mathbb R^{n+1}_+, \delta, x_{n+1}^mdx)$.
\end{conjecture}

\begin{remark}[On asymptotically flat SMMS]\label{rmk:as:flat}
It seems also natural to work with a SMMS $(M,g,v^mdvol)$ asymptotically flat in the following sense: there is a compact set $K\subset M$ such that $M\backslash K$ is diffeomorphic to $\{x\in\mathbb R^{n+1}_+\:|\:|x|>1\}$ and, in the induced coordinates,  $e_{ij}(x)=g_{ij}(x)-\delta_{ij}$, and $w(x)=\omega(x)x_{n+1}=v(x)-x_{n+1}$ have decay
$$
|e_{ij}(x)|+|x||e_{ij,k}(x)|+|x|^2|e_{ij,kl}(x)|\leq C|x|^{-\tau},\qquad i,j,k,l=1,...,n+1,
$$
and
$$
|\omega(x)|+|x||\omega_{,i}(x)|+|x|^2|\omega_{,ij}(x)|\leq C|x|^{-\tau},\qquad i,j=1,...,n+1,
$$
for some $\tau>(n+m-1)/2$ , with the further assumption that 
$$
\int_M |R_{(g,v,m)}|v^mdvol_g <\infty
\qquad\text{and}\qquad
\limsup_{s\to 0^+}\int_{\Sigma_s}|\mathcal H^{(s)}_{(g,v,m)}|s^md\sigma_g<\infty.
$$
In this case, a standard procedure would be to prove that the limit $m_{(g,v,m)}$ exists by linearizing $R_{(g,v,m)}$ and $\mathcal H_{(g,v,m)}$ using variations $(e,w)$ of $(g,v)$ around $(\delta,x_{n+1})$ in such a way that the charge integrands ${\mathbb U}(e,w)$ and $\overline {\mathbb U}(e,w)$ appear in divergence terms; see Section  \ref{sec:interpr:mass}.
Indeed, $R_{(g,v,m)}$ and $\mathcal H_{(g,v,m)}$  can be expressed as their respective linearizations plus quadratic terms in $(e,w)$ which are usually expected to be integrable thanks to the decay hypotheses. However, due to singular behavior of negative powers of $v$ near $\Sigma$, this approach does not seem to work for $m<1$ and that is the reason why we are imposing the stronger condition of $(M,g,v^mdvol)$ being asymptotically Schwarzschild. On the other hand, one could overcome this issue by simply assuming $m$ large enough so that the standard procedure would work to prove the existence of the limit $m_{(g,v,m)}$. However, we decided not to include this construction in this paper because our main interest here is the case $-1<m<1$. 
%and secondly because the hypothesis $m<1$ is necessary to prove the geometric invariance of $m_{(g,v,m)}$ under coordinates changes; see Section \ref{sec:geom:inv}. 
\end{remark}

\begin{remark}[On a warped product interpretation of mass]\label{rmk:warped}
Let $m=1,2,...$ and consider the model $$(\mathbb R^{n+1}_+,\delta,x_{n+1}^mdx).$$ 
The warped product metric $\delta+x_{n+1}^2g_{\mathbb S^m}$ on $\mathring{\mathbb R}^{n+1}_+\times \mathbb S^m$ is the Euclidean metric $\widetilde \delta=d\widetilde x_1^2+...+d\widetilde x_{n+m+1}^2$ on $\mathbb R^{n+m+1}$. This can be seen by the change of coordinates 
$$
x_1=\widetilde x_1,...,\: x_n=\widetilde x_n,\: x_{n+1}=y=\sqrt{\widetilde x_{n+1}^2+...+\widetilde x_{n+m+1}^2},
$$
for $x\in\mathring{\mathbb R}_+^{n+1}=\{x\in\mathbb R^{n+1}\:|\:x_{n+1}>0\}$ and $\widetilde x\in \mathbb R^{n+m+1}$, so that 
$$
\widetilde \delta=dx_1^2+...+dx_n^2+dy^2+y^2g_{\mathbb S^m}.
$$
Here, $g_{\mathbb S^n}$ stands for the canonical metric on the unit sphere $\mathbb S^n$.
Let $(M,g,v^mdvol_g)$ be an asymptotically flat SMMS as in Remark \ref{rmk:as:flat}. The warped product $\widetilde g=g+v^2g_{\mathbb S^m}$ on $\widetilde M=\mathring M\times \mathbb S^m$ yields an $(n+m+1)$-dimensional asymptotically flat Riemannian manifold $(\widetilde M,\widetilde g)$ in the usual sense, i.e., $\widetilde e_{ab}(\widetilde x)=\widetilde g_{ab}(\widetilde x)-\widetilde \delta_{ab}$ satisfies
$$
|\widetilde e_{ab}(\widetilde x)|+|\widetilde x||e_{ab,c}(\widetilde x)|+|\widetilde x|^2|\widetilde e_{ab,cd}(\widetilde x)|\leq C|\widetilde x|^{-\tau},\qquad a,b,c,d=1,...,n+m+1,
$$
with $R_{\widetilde g}$ integrable on $\widetilde M$. Indeed, $\widetilde e=e+(2\omega+\omega^2)y^2g_{\mathbb S^m}$,
where we are using the notation abuse
$$
e(\widetilde x)=e(\widetilde x_1,...,\widetilde x_n,y)
\qquad\text{and}\qquad
\omega(\widetilde x)=\omega(\widetilde x_1,...,\widetilde x_n,y).
$$
As observed in Remark \ref{rmk:as:flat}, $m_{(g,v,m)}$ is an existing limit for large $m$. A direct but long computation gives
\begin{align*}
\sum_{a,b=1}^{n+m+1}(\widetilde e_{ab,b}-\widetilde e_{bb,a})\widetilde x_a&=
\sum_{A=1}^{n}\mathbb U_A(e,y\omega)x_A+\mathbb U_{n+1}(e,y\omega)y
\\
&\hspace{0.5cm}-2m\sum_{A=1}^{n}\omega\omega_{,A}x_A-2m\omega\omega_{,n+1}y-m\omega^2,
\end{align*}
so the ADM mass of $(\widetilde M,\widetilde g)$ 
$$
\lim_{\rho\to \infty}\sum_{a,b=1}^{n+m+1}\int_{|\widetilde x|=\rho}(\widetilde e_{ab,b}-\widetilde e_{bb,a})\frac{\widetilde x_a}{\rho}d\widetilde \sigma
$$
equals a positive multiple of $m_{(g,v,m)}$ as the integrals involving the quadratic terms in $\omega$ vanish as $\rho\to\infty$. 
\end{remark}

The paper is organized as follows. In Sections 2 and 3, we collect definitions and basic results about the fractional Yamabe problem and smooth metric measure spaces. Our mass quantity is then defined in Section 4 where we also give the proofs of the first part of Theorem~\ref{thm:main} and Theorem~\ref{thm:second}. In Section 5 of the paper, we put the mass in a more general framework thus justifying the choice of its expression, while the mass' geometric invariance, i.e., the second part of Theorem~\ref{thm:main}, is proved in Section 6. Finally, some calculations used in Section 5 are obtained in the Appendix.

%%%%%%%%%%%%%%%%%%%%%%%%%%%%%%%%%%%%%%
\section{The fractional Yamabe problem}\label{sec:frac:yamabe}
The fractional Yamabe problem was proposed by Gonzalez and Qing in \cite{gonzalez-qing} as a generalization of the classical Yamabe problem by means of the fractional GJMS operators studied by Graham and Zworski in \cite{graham-zworski}. As mentioned in the introduction, they are non-local operators of fractional order defined using scattering theory. 
We begin with the concept of a conformally compact manifold.
\begin{definition}
We say that a Riemannian manifold $(X,g_+)$ is {\rm{conformally compact}} if there is a compact manifold  $\overline X$  with boundary such that $X$ is the interior of $\overline X$ and there is a neighborhood of the boundary $N$ of $\overline X$ such that the metric $g_+$ is expressed as 
$$
g_+=\rho^{-2}\overline g,
$$
for some metric $\overline g$ on $\overline X$ and some smooth non-negative function $\rho$ satisfying $\rho^{-1}(0)=N$ and $d\rho\neq 0$ everywhere on $N$. In this case, $\rho$ is said to be a {\rm{defining function}} of $(X,g_+)$. If in addition $\rho$ satisfies $|d\rho|_{\overline g}=1$ on $N$, $(X,g_+)$ is said to be {\rm{asymptotically hyperbolic}}.
\end{definition}

If $(X,g_+)$ is conformally compact, then it uniquely determines a conformal class of metrics on $N$, namely the class $[h]$, where $h=(\rho^2g_+)|_N$ for a defining function $\rho$. We say that $(N,[h])$ is the {\it{conformal infinity}} of $(X,g_+)$. 
The terminology asymptotically hyperbolic is justified by the fact that when $(X,g_+)$ is conformally compact then its Riemann tensor is expressed as 
$$
R_{ijkl}=-|d\rho|^2_{\overline g}\left((g_+)_{ik}(g_+)_{jl}-(g_+)_{il}(g_+)_{jk} \right)+O(\rho), \:\:\text{near}\: N,
$$  
where the $(0,4)$-Riemann tensor is defined in local coordinates by 
$$
R_{ijkl}=g(R(\partial_i,\partial_j)\partial_l,\partial_k).
$$
In particular, being asymptotically hyperbolic means that $g_+$ has negative constant sectional curvature up to a first order term, as $\rho\to 0$.

It is shown in \cite[Lemma 2.1]{graham} that, given a conformally compact asymptotically hyperbolic manifold $(X,g_+)$ and a representative $h$ on the conformal infinity, there is a unique defining function $\rho$ such that the metric $g_+$ is expressed as 
$$
g_+=\rho^{-2}(d\rho^2+h_\rho),\:\:\text{on a neighborhood of}\: N, 
$$
where $h_\rho$ is a one-parameter family of metrics on $N$ with $h_0=h$. Moreover, when $g_+$ is Einstein, $h_\rho$ has an asymptotic expansion which is even in powers of $\rho$ at least up to order $n$. In particular, $N$ is totally geodesic with respect to $\overline g=\rho^2g_+$.
This defining function is said to be {\it{geodesic}}.

The case of our main interest is when the conformally compact manifold $(X,g_+)$ of dimension $n+1$ is Einstein with negative Ricci curvature, i.e.,
$$
\text{Ric}_{g_+}=-ng_+.
$$
This particular case of an asymptotically hyperbolic manifold is called {\it{conformally compact Einstein}} (C.C.E. for short) or {\it {Poincar\'e-Einstein}}.

For the rest of this section, we assume that $(X,g_+)$ is C.C.E. with dimension $n+1$ and $\rho$ stands for its geodesic defining function.
In what follows, we briefly recall the construction of the scattering operators 
$$
S(s):C^\infty(N)\to C^\infty(N), \:\:s\in\mathbb C, \:\:\text{Re}\, s>n/2.
$$
We refer the reader to Section 2 of \cite{case-chang} and the references therein for details.
Assume that 
\begin{equation}\label{eq:s}
\begin{cases}
\text{Re}\,s>\frac{n}{2},\:\:
s\notin \frac{n}{2}+\mathbb N,
\\
s(n-s)\:\text{does not belong to the point spectrum of $-\Delta_{g_+}$.}
\end{cases}
\end{equation}
Consider the Poisson equation 
\begin{equation}\label{eq:poisson}
-\Delta_{g_+}u-s(n-s)u=0,\:\:\text{on}\:\:X.
\end{equation}
Given $f\in C^\infty(N)$, equation \eqref{eq:poisson} has a unique solution of the form 
\begin{equation}\label{eq:form:u}
u=F\rho^{n-s}+H\rho^s,\:\: F, H\in C^\infty(\overline{X}),\:\:F|_N=f,
\end{equation}
and $F$ expands as $F=f_{(0)}+f_{(2)}\rho^2+f_{(4)}\rho^4+...$, for $f_{(0)}=f$, where every other $f_{(2l)}$ is also determined by $f$.
The {\it{scattering operator}}, defined by $S(s)f=H|_N$, determines a family of meromorphic pseudo-differential operators in $\text{Re}\,s>n/2$ having poles corresponding to the elements of the point spectrum of $-\Delta_{g_+}$. 

Following Graham and Zworski in \cite{graham-zworski}, the scattering operators are used to define the {\it{fractional GJMS operators}} 
$$
P_\gamma=P_\gamma [g_+,h]:C^\infty(N)\to C^\infty(N),\:\:h=\overline g|_N,\:\:\overline g=\rho^2g_+,
$$ 
as follows. Let $\gamma\in(0,n/2)\backslash \mathbb N$ and assume that $s(n-s)$ does not belong to the point spectrum of $-\Delta_{g_+}$, where $s=\gamma+n/2$. In particular, $s$ satisfies the conditions \eqref{eq:s}. Define 
\begin{equation}\label{eq:P:d}
P_\gamma=d_\gamma S(s),\:\:d_\gamma=2^{2\gamma}\frac{\Gamma(\gamma)}{\Gamma(-\gamma)}.
\end{equation}
The operators $P_\gamma$ so defined are formally self-adjoint and have principal symbol $|\xi|^{2\gamma}$. They extend, to fractional values of $\gamma$, the operators discovered by Graham, Jenne, Mason and Sparling in \cite{GJMS} for $\gamma\in\mathbb N$. The following conformal properties hold:
$$
P_\gamma[g_+,w^{\frac{4}{n-2\gamma}}h](u)=w^{-\frac{n+2\gamma}{n-2\gamma}}P_\gamma[g_+,h](wu).
$$
Due to further application to the fractional Yamabe problem and for simplicity of statements, we will restrict our analysis in this section to the case $\gamma\in (0,1)$. 

Chang and Gonzalez found in \cite{chang-gonzalez} an alternate way of describing the operators $P_\gamma$ in terms of an extension degenerate elliptic problem, generalizing a result by Caffarelli and Silvestre in \cite{caffarelli-silvestre}. 
Equation \eqref{eq:poisson} is equivalently written in terms of the compactified metric $\overline g=\rho^2g_+$ on $\overline X$ as
$$
-\text{div}_{\overline g}(\rho^m\nabla_{\overline g}U)+E(\rho)U=0,\:\:\text{in}\:\:X,
$$
where $m=1-2\gamma=n+1-2s$, $U=\rho^{s-n}u$ and 
$$
E(\rho)=-\rho^{\frac{m}{2}}\Delta_{\overline g}\rho^{m/2}+\frac{1}{4}m(m-2)\rho^{m-2}+\frac{n-1}{4n}R_{\overline g}\rho^m.
$$
The fractional GJMS operator is expressed as 
$$
P_\gamma \left(U|_N\right)=\frac{d_\gamma}{2\gamma}\lim_{\rho\to 0}(\rho^m\partial_\rho U).
$$

Now we turn to the fractional Yamabe problem, proposed in \cite{gonzalez-qing}.  
In analogy with the scalar curvature, the {\it{fractional scalar curvature}} is defined by  $P_\gamma(1)$.
The fractional Yamabe problem consists in finding a positive solution to the non-local non-linear problem 
$$
P_\gamma w = c w^{\frac{n+2\gamma}{n-2\gamma}},\:\:\text{on}\:\: N,
$$
for some $c\in\mathbb R$. The seeking solution $w>0$ represents a conformal metric $w^{\frac{4}{n-2\gamma}}h$ with constant fractional scalar curvature and the limit case $\gamma=1$ becomes the classical Yamabe problem. The formulation of the fractional Yamabe problem in terms of the extension problem  
$$
\begin{cases}
-\text{div}_{\overline g}(\rho^m\nabla_{\overline g}U)+E(\rho)U=0,\:\:\text{in}\:\:X,
\\
\lim_{\rho\to 0}(\rho^m\partial_\rho U)=cU^{\frac{n+2\gamma}{n-2\gamma}}, \:\:\text{on}\:\: N,
\end{cases}
$$
where $c\in \mathbb R$, makes a connection with the Escobar-Yamabe problem in \cite{escobar2} which stands for the case $\gamma=1/2$ in this setting.

The fractional Yamabe problem admits a variational formulation for the extension problem in terms of the functional
$$
U\mapsto 
\frac{\int_X \left( \rho^{1-2\gamma}|\nabla_{\overline g}U|^2_{\overline g}+E(\rho)U^2\right)dvol_{\overline g}}
{\left(\int_{N} |U|^{\frac{2n}{n-2\gamma}}d\sigma_{\overline g}\right)^{\frac{n-2\gamma}{n}}}
$$
whose infimum is the conformal invariant denoted by $\Lambda^\gamma(N,[\overline g|_N])$. An Aubin type existence result holds for this setting (see \cite[Proposition B]{kim-musso-wei}). 
It gives the existence of a minimizing positive solution when $\Lambda^\gamma(N,[\overline g|_N])$ is finite and strictly less than the one for the standard hemisphere.

The existence of a Green's function for this setting is proved by Kim, Musso and Wei in \cite[Proposition 1.5]{kim-musso-wei}. We present here the (slightly more restrictive) statement for C.C.E. manifolds with geodesic defining functions.

\begin{proposition}[\cite{kim-musso-wei}]\label{propo:green}
Suppose that $\Lambda^\gamma(N,[\overline g|_N])>0$ and the first $L^2$-eigenvalue of the Laplace-Beltrami operator satisfies 
$$
\lambda_1(-\Delta_{g_+})>\frac{n^2}{4}-\gamma^2.
$$ 
Then for each $p\in N$ there exists a unique positive Green's function $G_p$ on $\overline X\backslash \{p\}$ which satisfies 
\begin{equation}\label{system:green}
\begin{cases}
-\text{div}_{\overline g}(\rho^m\nabla_{\overline g}G_p)+E(\rho)G_p=0,\:\:\text{in}\:\: X,
\\
-\lim_{\rho\to 0}(\rho^m\partial_\rho G_p)=\delta_p, \:\:\text{on}\:\: N,
\end{cases}
\end{equation}
in the distributional sense, where $\delta_p$ is a positive multiple of Dirac measure.
\end{proposition} 

Inspired by Schoen's argument for the classical Yamabe problem in \cite{schoen}, 
Green's functions are used in the construction of test functions aimed to prove an Aubin inequality on the fractional Yamabe problem. 

\begin{conjecture}[\cite{kim-musso-wei}]\label{conj:green}
Suppose that $\Lambda^\gamma(N,[\overline g|_N])>0$ and either $(N,[\overline g|_N])$ is locally conformally flat or $n=2$. Then $G_p$ has the asymptotic expansion
\begin{equation}\label{eq:green}
G_p(q)=d_{\overline g}(q,p)^{2\gamma-n}+A+O'(d_{\overline g}(q,p)^{\min\{1,2\gamma\}}),
\end{equation}
for $q\in\overline X$ near $p$, where $A\in\mathbb R$, and $d_{\overline g}(\cdot,\cdot)$ is the distance function. Furthermore, $A\geq 0$ with equality holding if and only if $(\overline X,\overline g)$ is conformally diffeomorphic to the standard unit ball in $\mathbb R^{n+1}$. 
\end{conjecture}

This conjecture is analogous to the positive mass theorem by Schoen and Yau for locally flat manifolds established in \cite{schoen-yau2}. In that case, there is a connection of the term $A$ in the Green's function expansion and the ADM mass of the asymptotically flat manifold obtained by the generalized stereographic projection which was done by Schoen in the solution of the Yamabe problem in \cite{schoen}. 
Kim, Musso and Wei proved that if Conjecture \ref{conj:green} is valid then an Aubin inequality holds when $N$ is locally conformally flat so that the fractional Yamabe problem is solvable in this case. The expansion \eqref{eq:green} was obtained by Mayer and Ndiaye in \cite{mayer-ndiaye1} for $\gamma\neq 1/2$ (without determining the sign of the term $A$), where they obtain $C^{2,\alpha}$ regularity for $G_p$ away from the point $p$. 
A proof of the Yamabe problem in the case of locally flat conformal infinity (without the sign hypothesis on $A$) was presented in \cite{mayer-ndiaye2} using Bahri-Coron methods. The authors also define the "fractional mass map" which takes each point $p\in N$ to the corresponding constant $A$ (see \cite[Definition 4.3]{mayer-ndiaye2}).     
Conjecture \ref{conj:green} was also used in \cite{an-lacroix-chan-ho, chan-sire-sun} to prove the convergence of the fractional Yamabe flow.

In this paper, we give a geometrical interpretation for the term $A$ in the Green's function expansion \eqref{eq:green} by analogy with the ADM mass in general relativity.
This will be done in Section \ref{sec:mass:def} by making use of the concept of smooth metric measure spaces.

%%%%%%%%%%%%%%%%%%%%%%%%%%%%%%%%%%%%%%%%%%%%%%%%%%%%%%%%%%%%%%%%%%%%%%%%%%%%%%%%%%%

%%%%%%%%%%%%%%%%%%%%%%%%%%%%%%%%%%%%%%%%%%%%%%%

\section{Smooth metric measure spaces}\label{sec:smms}

Following the definitions in \cite{case, case-chang}, a {\it{smooth metric measure space}}, or SMMS for short, is a four-tuple $(M,g,v^m dvol_g, \mu)$, 
where $M$ is a $(n+1)$-dimensional manifold with boundary $\Sigma$, $g$ is a Riemannian metric on $M$ with volume element $dvol_g$, 
$v$ is a non-negative smooth function on $M$ satisfying $v^{-1}(0)=\Sigma$, $m\in\mathbb R\cup\{\pm\infty\}$ and $\mu\in\mathbb R$.
We will not require here $M$ to be compact. Actually, our model example in the next section will be 
$$
\mathbb R^{n+1}_+=\{(x_1,...,x_{n+1})\in\mathbb R^{n+1}\: | \: x_{n+1}\geq 0\}
$$
with the Euclidean metric $\delta$, and with $v$ being given by the coordinate function $x_{n+1}$. We will restrict our analysis to the case $m\in (1-n,\infty)$.  
%In many situations, it will be convenient to write the measure $v^m dvol_g$ as $\text{e}^{-\phi}dvol_g$, where $\phi$ is a smooth function on the interior of $M$.
%For consistence with some examples in Section \ref{sec:prel}, we will also allow the boundary $\Sigma$ to be empty.

As observed by Case and Chang in \cite{case-chang}, formally one can think of a SMMS $(M, g, v^m dvol_g, \mu)$ as the base of the warped product $(M\times F_\mu^m, g\oplus v^2 g_F)$ for $(F_\mu^m, g_F)$ the simply-connected $m$-dimensional space-form with Ricci curvature $\text{Ric}(g_F)=\mu g_F$. Geometric invariants  on the warped product when $m\in\mathbb N$ are restricted to the base and this construction is extended to general $m$ by treating $m$ as a formal parameter. This process is the one used to obtain the expression \eqref{eq:R} below for the weighted scalar curvature $R_{(g,v,m)}$ which plays a major role in our work.

\vspace{0.2cm}
\noindent
{\bf{Agreement.}}
Throughout the paper, we will restrict our analysis to the case $\mu=m-1$ and consider SMMS  $(M,g,v^m dvol_g, m-1)$ denoting them as triples $(M,g,v^m dvol_g)$ for short. 
\vspace{0.2cm}

If $R_g$ is the scalar curvature of $(M,g)$, the {\it{weighted scalar curvature}} of the SMMS $(M,g,v^m dvol_g)$ is given by the expression
\begin{equation}\label{eq:R}
R_{(g,v,m)}=R_g-2mv^{-1}\Delta_g v -m(m-1)v^{-2}|\nabla_g v|_g^2+m(m-1) v^{-2}.
\end{equation}
This quantity is defined in the interior of $M$, where $v>0$.
%For small $\epsilon\geq 0$, we define the hypersurface $\Sigma_\epsilon=v^{-1}(\epsilon)$ so that $\Sigma_0=\Sigma$ is the boundary of $M$.
%If $x\in\Sigma_\epsilon$,  we denote by $\eta(x)$ the outwards pointing unit normal vector to $\Sigma_\epsilon$ at $x$, and by $H_g(x)=\text{div}_g\eta(x)$ the mean curvature of  $\Sigma_\epsilon$ at $x$. The {\it{weighted mean curvature}} (see \cite{gromov}) of $\Sigma_\epsilon$ is $H_{(g,v,m)}=H_g+mv^{-1}\eta(v)$, $\epsilon>0$. In this paper, we will make use of the following {\it{modified weighted mean curvature}} of $\Sigma_\epsilon$, for $\epsilon>0$, 
The {\it{weighted mean curvature}} (see \cite{gromov}) of a given hypersurface $S\subset \{v>0\}$ is $$H_{(g,v,m)}=H_g+mv^{-1}\eta(v),$$ where $\eta$ is the outward pointing unit normal vector. In this paper, we define the {\it{modified weighted mean curvature}} by the expression
\begin{equation}\label{eq:H}
\mathcal H_{(g,v,m)}=H_g+mv^{-1}\eta(v)+mv^{-1}.
\end{equation}
\begin{remark}[A geometrical interpretation of $\mathcal H_{(g,v,m)}$]\label{rmk:inter:H}
Observe that $\mathcal H_{(g,v,m)}-H_{(g,v,m)}=mv^{-1}$. The term $mv^{-1}$ can be explained as follows. In the case of the Euclidean half-space with $v(x)=x_{n+1}$, we have $H_{(\delta,v,m)}=-mv^{-1}$ on any horizontal hyperplane of $\mathbb R^{n+1}_+$. On those hyperplanes,
$$
\mathcal H_{(g,v,m)}=H_{(g,v,m)}-H_{(\delta,v,m)}.
$$  
Hence, the modified weighted mean curvature can be viewed as a re-normalization of the weighted mean curvature obtained by subtracting the divergent contribution of the Euclidean model. Thus, the dominant energy condition $\liminf_{s\to 0^+}s^m\mathcal H^{(s)}_{(g,v,m)}\geq 0$ in Conjecture \ref{our:conj} has a parallel with that of \cite[Theorem 1.9]{almaraz-lima2} for example. This re-normalization is also precisely the one for which the boundary linearization couples with that of the weighted scalar curvature in the Green-type identities developed in Section \ref{sec:interpr:mass}.
\end{remark}
%\begin{remark}\marginnote{This remark can be removed now}
%The term $v^{-1}$ is not well defined in the case where the given hypersurface is the boundary $\Sigma=v^{-1}(0)$. However, when this quantity is used below in Section \ref{sec:interpr:mass}, we will assume $\nabla_gv=-\eta$ so that $\mathcal H_{(g,v,m)}$ becomes $H_g$. This holds for our model  $(\mathbb R^n_+, \delta, x_{n+1}^m dx)$ which serves as the underlying SMMS in our construction. Nevertheless, the complete expression \eqref{eq:H}  should be preserved since we will be interested in variations of $\mathcal H_{(g,v,m)}$ with respect to $g$ and $v$.
%\end{remark}
%\begin{remark}
%It is of interest the case when $\nabla_gv=-\eta$ on $\Sigma_\epsilon$ as the model $$(\mathbb R^n_+, \delta, x_{n+1}^m dx)$$ satisfies this condition. In this case, $\mathcal H_{(g,v,m)}=\mathcal H_{(\delta, x_{n+1}, m)}=H_\delta=0$ on $\Sigma_\epsilon$. 
%The extra-term $mv^{-1}$, added to the weighted mean curvature, is necessary for the calculations in Section \ref{sec:interpr:mass} and Appendix A (where the condition $\nabla_gv=-\eta$ on $\Sigma_\epsilon$ is used to ensure a certain cancellation).
%\end{remark}
The {\it{weighted conformal Laplacian}} (see \cite{case}) is defined by the expression
\begin{equation}\label{eq:Lv}
L_{(g,v,m)}w=-\Delta_g w-mv^{-1}g(\nabla_g v, \nabla_g w)+\frac{m+n-1}{4(m+n)}R_{(g,v,m)}w
\end{equation}
where $w$ is a smooth function on $M$.
It will also be interesting to define the {\it{weighted boundary conformal operator}}
\begin{equation}\label{eq:Bv}
B_{(g,v,m)}=-\frac{\partial}{\partial \eta}-\frac{m+n-1}{2(m+n)}\mathcal H_{(g,v,m)}.
\end{equation}
%This operator will not be used in the paper but it is included in here for potential future applications.

We say that another SMMS $(M,\hat g, \hat v^mdvol_{\hat g})$ is conformally equivalent to $(M, g, v^mdvol_g)$ if there is a positive smooth function $u$ on $M$ satisfying $\hat g=u^{-2}g$ and $\hat v=u^{-1}v$. In this case, the following equations are direct computations:
\begin{equation}\label{eq:conf:L}
L_{(\hat g, \hat v,m)}(w)=u^{\frac{m+n+3}{2}}L_{(g,v,m)}(u^{-\frac{m+n-1}{2}}w),
\end{equation}
\begin{equation}\label{eq:conf:B}
B_{(\hat g, \hat v,m)}(w)=u^{\frac{m+n+1}{2}}B_{(g,v,m)}(u^{-\frac{m+n-1}{2}}w),
\end{equation}
for any smooth function $w$ on $M$. Formula \eqref{eq:conf:L} is \cite[Equation (3.1)]{case}, while \eqref{eq:conf:B} can be verified using the well-known formula $H_{\hat g}=uH_g-n \partial u/\partial \eta$.
%\cite[Equation (1.8)]{escobar2}.

Let $(X,g_+)$ be a C.C.E. manifold with conformal infinity $N$ and geodesic defining function $\rho$. If $(\overline X,\overline g=\rho^2g_+)$ is the compactification of $(X,g_+)$, the SMMS $(X,g_+,1^mdvol_{g+})$ is conformally equivalent to $(\overline X,\overline g,\rho^mdvol_{\overline g})$. The formulation of the fractional Yamabe problem in terms of smooth metric measure spaces was done in \cite{case-chang} as follows.
Set $m=1-2\gamma$. According to \cite[Theorem 4.1]{case-chang}, given $f\in C^\infty(N)$, a function $U$ is a solution of 
\begin{equation}\label{eq:L:U}
\begin{cases}
L_{(\overline g,\rho,m)}U=0, &\text{in}\:X,
\\
U=f, &\text{on}\:N,
\end{cases}
\end{equation}
if and only if $u=\rho^{n-s}U$ is the solution of equation \eqref{eq:poisson} of the form \eqref{eq:form:u} with $s=\gamma+n/2$. 
Moreover, $U$ satisfies 
$$
P_\gamma f=\frac{d_\gamma}{2\gamma}\lim_{\rho\to 0}\left(\rho^m\frac{\partial U}{\partial \rho}\right),
$$ 
for $d_\gamma$ as in \eqref{eq:P:d}.
Indeed,  
\begin{equation}\label{eq:L:div}
\rho^{m}L_{(\overline g, \rho,m)}U=-\text{div}_{\overline g}(\rho^m\nabla_{\overline g}U)+E(\rho)U.
\end{equation}
The fractional Yamabe problem reads as
$$
\begin{cases}
L_{(\overline g,\rho,m)}U=0,\:\:\text{in}\:\:X,
\\
\lim_{\rho\to 0}(\rho^m\partial_\rho U)=cU^{\frac{n+2\gamma}{n-2\gamma}}, \:\:\text{on}\:\: N.
\end{cases}
$$

%%%%%%%%%%%%%%%%%%%%%%%%%%%%%%%%%%%%%%%%%%%%%%%%%%%%%%%%%%%%%%%%%%%%%%%%%%%%%%%%%

\section{The mass definition}\label{sec:mass:def}

In this section, we use SMMS to define the mass-type invariant mentioned in the introduction section. This extends the mass in \cite{almaraz-barbosa-lima} and gives a geometrical interpretation for the constant $A$ in equation \eqref{eq:green:0} (or \eqref{eq:green}).

%In this section, we will extend the geometric mass-type invariant studied in \cite{almaraz-barbosa-lima} to SMMS asymptotically modelled on $(\mathbb R^{n+1}_+,\delta,x_{n+1}^m dx,\mu)$ in such a way that the one in \cite{almaraz-barbosa-lima} corresponds to the case $m=0$. These invariants are strongly related to the classical ADM mass on asymptotically flat manifolds defined by Arnowitt, Deser and Misner in \cite{arnowitt-deser-misner}. As discussed below, the mass invariant defined here for SMMS gives some geometrical meaning for the key constant term $A$ appearing in the Green's function expansion in Conjecture \ref{conj:green}.

%We will restrict our analysis to the case $\mu=m-1$ and consider SMMS  $(M,g,v^m dvol_g, m-1)$ denoting them as triples $(M,g,v^m dvol_g)$ for short. 

\begin{definition}\label{def:as:flat}
Let $(M,g,v^m dvol_g)$ be a SMMS. Suppose that there is a compact set $K\subset M$ and a diffeomorphism 
$$
\phi:M\backslash K\to \{x\in\mathbb R^{n+1}_+\:|\:|x|> 1\}
$$
such that, in the induced coordinates $x=(x_1,...,x_{n+1})$, 
there is $0<\epsilon<1$ such that, 
\begin{equation}\label{eq:can:exp}
		(\phi_*g)_{ij}(x)=\left(1+a|x|^{1-n-m}\right)\delta_{ij}+O'(|x|^{\epsilon-n-m}),\quad i,j=1,...,n+1,
\end{equation} and 
\begin{equation}\label{eq:can:exp:v}
		v\circ\phi^{-1}(x)=\left(1+\displaystyle\frac{a}{2}|x|^{1-n-m} \right)x_{n+1}+x_{n+1}O'(|x|^{\epsilon-n-m}), 
\end{equation}
for some constant $a$ and any  $x\in \{x\in\mathbb R^{n+1}_+\:|\:|x|> 1\}$. In this case, we say that 
$(M,g,v^m dvol_g)$ is an asymptotically Schwarzschild SMMS with a non-compact boundary and denote it by ASSMMS for short.
\end{definition}

\begin{remark}\label{rmk:AS}
The expansions in Definition~\ref{def:as:flat} are equivalent to 
\begin{align*}
	\begin{cases}
		(\phi_*g)_{ij}(x)=\left(1+\displaystyle\frac{(m+n-1)a}{4}|x|^{1-n-m}\right)^{\frac{4}{m+n-1}}\delta_{ij}+O'(|x|^{\epsilon-n-m}),
		\\
		v\circ\phi^{-1}(x)=\left(1+\displaystyle\frac{(m+n-1)a}{4}|x|^{1-n-m} \right)^{\frac{2}{m+n-1}}x_{n+1}+x_{n+1}O'(|x|^{\epsilon-n-m}).
	\end{cases}
\end{align*}
This suggests that an ASSMMS being "exactly" Schwarzschild means that the error terms $O'$ above are zero.
So, it is natural to consider the SMMS $(\mathbb R^{n+1}_+\backslash\{0\}, g_a, v_a^mdvol_{g_a})$, where $a>0$ and 
\begin{align*}
	\begin{cases}
		g_a(x)=\left(1+\displaystyle\frac{(m+n-1)a}{4}|x|^{1-n-m}\right)^{\frac{4}{m+n-1}}\delta,
		\\
		v_a(x)=\left(1+\displaystyle\frac{(m+n-1)a}{4}|x|^{1-n-m} \right)^{\frac{2}{m+n-1}}x_{n+1}.
	\end{cases}
\end{align*}
Observe that when $m=0$,  $g_ a$ becomes the classical Schwarzschild metric in general relativity. Direct calculations show that $R_{(g_a,v_a,m)}=0$ and $\lim_{s\to 0^+}s^m\mathcal H_{(g_a,v_a,m)}^{(s)}=0$ where $\mathcal H_{(g_a,v_a,m)}^{(s)}$ stands for the modified weighted mean curvature of the horizontal level set $x_{n+1}=s>0$. On the other hand, we can also prove that the weighted mean curvature $H_{(g_a,v_a,m)}$ vanishes on the hemisphere of symmetry $\{|x|=\left(\frac{m+n-1}{4}a\right)^{\frac{1}{m+n-1}}\}$.
\end{remark}

\begin{definition}\label{def:mass}
The {\rm{mass}} of an ASSMMS $(M,g,v^m dvol_g)$ is 
\begin{align}\label{eq:mass}
\mathfrak m_{(g,v,m)}=\lim_{\rho\to\infty}\lim\limits_{s\to 0^+}\Big \{
&\int_{|x|=\rho, x_{n+1}>s}\sum_{i=1}^{n+1}\mathbb U_i(\phi_*g-\delta, v\circ\phi^{-1}-x_{n+1}) \frac{x_i}{|x|}x_{n+1}^md\sigma
\\
&+\int_{|x|=\rho, x_{n+1}= s}\sum_{\alpha=1}^{n}\overline {\mathbb U}_\alpha(\phi_*g-\delta, v\circ\phi^{-1}-x_{n+1}) \frac{x_\alpha}{|x|}x_{n+1}^m d\overline \sigma
\Big \},\notag
\end{align}
where 
$$
\mathbb U(e,w)=\text{div}_\delta e-d\text{tr}_\delta e+mx_{n+1}^{-1}\nabla_\delta x_{n+1}\righthalfcup e-2mx_{n+1}^{-1} dw
$$ 
and $\overline {\mathbb U}(e,w)=\nabla_\delta x_{n+1}\righthalfcup e$ are the {\rm{charge integrand}} one-forms, and $d\sigma$ and $d\overline\sigma$ are Euclidean area elements.
\end{definition}

First let us calculate the mass quantity in view of the asymptotic expansion in Definition~\ref{def:as:flat}:
\begin{proposition}\label{lemma:mass}
    If  $0\leq m<1$, or $-1<m<0$ with $O'(|x|^{\epsilon-n-m})=O'(|x|^{\epsilon-n-m})\delta_{ij}+x_{n+1}O'(|x|^{\epsilon-n-m-1})$ in \eqref{eq:can:exp}, then
	$$
	\mathfrak m_{(g,v,m)}=(m+n)(m+n-1)a \sigma_{n,m},
	$$
	where 
	$$
	\sigma_{n,m}=\int_{|x|=1, x_{n+1}> 0}x_{n+1}^m d\sigma.
	$$
\end{proposition}
\begin{proof} 
Set $e=\phi_*g-\delta$ and $w=v\circ\phi^{-1}-x_{n+1}$ outside a compact set of $M$. 	First suppose that $m>0$. Then
	$$e=a|x|^{1-n-m}\delta+O'(|x|^{\epsilon-n-m}),$$ and
	$$w=\frac{a}{2}|x|^{1-n-m} x_{n+1}+x_{n+1}O'(|x|^{\epsilon-n-m}).$$
	Observe that  
	$$\sum_{\alpha=1}^{n}\overline {\mathbb U}_\alpha(e, w)\frac{x_\alpha}{|x|}=O(|x|^{\epsilon-n-m})$$ 
	and 
	$$\sum_{i=1}^{n+1}\mathbb U_i(e, w) \frac{x_i}{|x|}=\sum_{i,j=1}^{n+1}(e_{ij,j}-e_{jj,i}+mx_{n+1}^{-1}(e_{n+1\,i}-2w_{,i}))\frac{x_i}{|x|}+O(|x|^{\epsilon-n-m-1}).$$
	Direct calculations show that
	$$\sum_{i,j=1}^{n+1}e_{ij,j}\frac{x_i}{|x|}=(1-m-n)a|x|^{-n-m}+O(|x|^{\epsilon-n-m-1}),$$
	and 
	$$\sum_{i,j=1}^{n+1}e_{jj,i}\frac{x_i}{|x|}=(n+1)(1-m-n)a|x|^{-n-m}+O(|x|^{\epsilon-n-m-1})$$
	so that
	$$\sum_{i,j=1}^{n+1}(e_{ij,j}-e_{jj,i})\frac{x_i}{|x|}=n(m+n-1)a|x|^{-n-m}+O(|x|^{\epsilon-n-m-1}).$$
	On the other hand, 
	\begin{equation}\label{eq:mass:1}
	\sum_{i=1}^{n+1}mx_{n+1}^{-1}e_{n+1\,i}\frac{x_i}{|x|}=ma|x|^{-n-m}+O(x_{n+1}^{-1}|x|^{\epsilon-n-m}),
	\end{equation}
	and 
	\begin{equation}\label{eq:mass:2}
	\sum_{i=1}^{n+1}mx_{n+1}^{-1}(-2w_{,i})\frac{x_i}{|x|}=m(m+n-2)a|x|^{-n-m}+O(|x|^{\epsilon-n-m-1}).
	\end{equation}
	Adding all these together we have
	\begin{equation}\label{eq:mass:3}
	\sum_{i=1}^{n+1}\mathbb U_i(e, w) \frac{x_i}{|x|}=(m+n)(m+n-1)a|x|^{-n-m}+O(|x|^{\epsilon-n-m-1})+O(x_{n+1}^{-1}|x|^{\epsilon-n-m}).
	\end{equation}
	Moreover
$$
\lim\limits_{s\to 0^+}\int_{|x|=\rho, x_{n+1}=s}\sum_{\alpha=1}^{n}\overline {\mathbb U}_\alpha(e, w)\frac{x_\alpha}{|x|}x_{n+1}^m d\overline\sigma=0
$$  	
so that 
$$
\mathfrak m_{(g,v,m)}=\lim_{\rho\to\infty}\left((m+n)(m+n-1)a \sigma_{n,m}+\sigma_{n,m-1}O(\rho^{\epsilon-1})+\sigma_{n,m}O(\rho^{\epsilon-1})\right).
$$
The result follows as $\epsilon<1$ and $\sigma_{n,m}$ and $\sigma_{n,m-1}$ are finite integrals because $m>0$.

If $m=0$ the left sides of \eqref{eq:mass:1} and \eqref{eq:mass:2} vanish and equation \eqref{eq:mass:3} simplifies to 
\begin{equation*}
	\sum_{i=1}^{n+1}\mathbb U_i(e, w) \frac{x_i}{|x|}=n(n-1)a|x|^{-n}+O(|x|^{\epsilon-n-1}).
\end{equation*}
The result now follows similarly as above by observing that 
$$\sum_{\alpha=1}^{n}\overline {\mathbb U}_\alpha(e, w)\frac{x_\alpha}{|x|}=O(|x|^{\epsilon-n}).$$ 
Finally, when $-1<m<0$, by the assumption we have 
	$$e=a|x|^{1-n-m}\delta+O'(|x|^{\epsilon-n-m})\delta+x_{n+1}O'(|x|^{\epsilon-n-m-1}),$$ and
	$$w=\frac{a}{2}|x|^{1-n-m} x_{n+1}+x_{n+1}O'(|x|^{\epsilon-n-m}).$$
One can calculate similarly as the case $0<m<1$ to obtain
$$
	\sum_{i=1}^{n+1}\mathbb U_i(e, w) \frac{x_i}{|x|}=(m+n)(m+n-1)a|x|^{-n-m}+O(|x|^{\epsilon-n-m-1}).$$
and
$$
\lim\limits_{s\to 0^+}\int_{|x|=\rho, \,x_{n+1}=s}\sum_{\alpha=1}^{n}\overline {\mathbb U}_\alpha(e, w)\frac{x_\alpha}{|x|}x_{n+1}^md\overline\sigma=0.
$$  	
Once again the result follows.

\end{proof}

This is the first part of Theorem \ref{thm:main}.
Its second part comes from the following:
\begin{theorem}\label{thm:geom:inv}
If  $0\leq m<1$, or $-1<m<0$ with $O'(|x|^{\epsilon-n-m})=O'(|x|^{\epsilon-n-m})\delta_{ij}+x_{n+1}O'(|x|^{\epsilon-n-m-1})$ in \eqref{eq:can:exp}, the quantity $\mathfrak m_{(g,v,m)}$ in \eqref{eq:mass}  is independent of the chosen asymptotic coordinates in the following sense: if $y=(y_1,...,y_{n+1})$ is another set of asymptotic coordinates coming from a diffeomorphism $\widetilde\phi$ such that 
	\begin{equation}\label{eq:exp:y}
		(\widetilde\phi_*g)_{ij}(y)=\left(1+\widetilde a|y|^{1-n-m}\right)\delta_{ij}+O'(|y|^{\epsilon-n-m}),
	\end{equation}
	 and 
	\begin{equation}\label{eq:exp:y:v}
		v\circ\widetilde\phi^{-1}(y)=\left(1+\displaystyle\frac{\widetilde a}{2}|y|^{1-n-m} \right)y_{n+1}+y_{n+1}O'(|y|^{\epsilon-n-m}),
\end{equation}
with $O'(|y|^{\epsilon-n-m})=O'(|y|^{\epsilon-n-m})\delta_{ij}+y_{n+1}O'(|y|^{\epsilon-n-m-1})$ in~\eqref{eq:exp:y} when $-1<m<0$,
then $\widetilde a=a$. % so that  $\mathfrak m_{(g,v,m)}$ remains unchanged.
\end{theorem}

The independence of the asymptotic coordinates characterizes $\mathfrak m_{(g,v,m)}$ as a geometric invariant of the ASSMMS. 
Theorem \ref{thm:geom:inv} is proved in Section \ref{sec:geom:inv}.

\begin{remark}\label{rmk:baldauf-ozuch}
A mass-type geometric quantity was studied in \cite{baldauf-ozuch} for weighted manifolds without boundary, which corresponds to the case of SMMS with $\mu=0$ and $m=\infty$. Their authors also proved corresponding positive mass theorems under the spin assumption. Those results were extended in \cite{chu-zhu} to the case $3\leq n\leq 7$. In \cite{law-lopez-santiago}, this mass quantity is related to a mass-type geometric invariant for a SMMS in the case where $\mu=0$ by considering the SMMS as the base of a warped product manifold.
\end{remark}

%The case of main interest comes from the generalized stereographic projetion, constructed below in Section \ref{sec:stereogr}.  In this case, $\mathfrak m_{(g,v,m)}=c(n,m)A$ where $c(n,m)>0$ and $A$ is the constant in the expansion of the Green's function in Conjecture \ref{conj:green}. This is the content of Theorem \ref{propo:as:flat} below.

%As mentioned above, Conjeture \ref{conj:green} is proved by Wang-Zhou in \cite{wang-zhou}, establishing a positive mass theorem for this setting. 

Inspired by Schoen's classical construction in \cite{schoen}, we define the generalized stereographic projection for a C.C.E. manifold as follows.
Let $(X,g_+)$ be a C.C.E. manifold with conformal infinity $N=\partial \overline X$ and let $\rho$ be its geodesic defining function.

\begin{definition}\label{def:ster:proj}
Let $0<\gamma<1$.
Fix $p\in N$ and suppose that $G_p$ is a positive Green's function satisfying the system \eqref{system:green}.
Define the {\rm{stereographic projection}} of 
$$(\overline X, \overline g, \rho^m dvol_{\overline g})$$ 
as the SMMS $(M,g,v^mdvol_g)$, where $m=1-2\gamma$, $M=\overline X\backslash\{p\}$, $g=G_p^{\frac{4}{n-2\gamma}}\overline g$ and $v=G_p^{\frac{2}{n-2\gamma}}\rho$. 
\end{definition} 

The following result shows that the ASSMMS in Definition \ref{def:as:flat} comprises the stereographic projection in Definition \ref{def:ster:proj}:

\begin{theorem}\label{propo:as:flat}
Suppose that $n=2$ or $N$ is locally conformally flat, and fix $0<\gamma<1$, $\gamma\neq 1/2$. 
Let $p\in N$ and choose the compactified metric $\overline g$ in such a way that $\overline g|_N$ is flat near $p$.
The stereographic projection  $(M,g,v^mdvol_g)$ of $(\overline X, \overline g, \rho^m dvol_{\overline g})$ is an ASSMMS, with non-compact boundary $\Sigma=N\backslash\{p\}$, satisfying the asymptotic expansion \eqref{eq:can:exp} and \eqref{eq:can:exp:v} with 
$$a=\frac{4A}{n-2\gamma}.$$
Then the mass of $(M,g,v^mdvol_g)$ is given by 
$$
\mathfrak m_{(g,v,m)}=4A(1-2\gamma+n)\int_{|x|=1, x_{n+1}> 0}x_{n+1}^{1-2\gamma}d\sigma.
$$
Moreover, $R_{(g,v,m)}=0$ in $M\backslash\Sigma$ and $\lim_{s\to 0^+} s^m\mathcal H^{(s)}_{(g,v,m)}=0$ uniformly on compact sets of $\Sigma$, where $\mathcal H^{(s)}_{(g,v,m)}$ is the modified weighted mean curvature on $\Sigma_s=v^{-1}(s)$. 
\end{theorem}
\begin{proof}
It follows from \cite[Equation (47)]{mayer-ndiaye1} for $n=2$ and \cite[Equation (4-7) and Lemma 4.3]{kim-musso-wei} for $n\geq 3$ that 
\begin{equation}\label{eq:coord}
\overline{g}\left(\frac{\partial}{\partial y_j},\frac{\partial}{\partial y_l}\right)=\delta_{jl}+O(y_{n+1}^n),\:\: j,l=1,...,n+1,
\end{equation}
in an appropriate coordinate system $(y_1,...,y_{n+1})$ such that $y_{n+1}=\rho(y)$. 
For inverted coordinates $x=|y|^{-2}y$ we have 
$$
\frac{\partial}{\partial x_i}=\sum_{j=1}^{n+1}|x|^{-2}(\delta_{ij}-2|x|^{-2}x_ix_j)\frac{\partial}{\partial y_j},\:\:i=1,...,n+1,
$$
so that
\begin{equation}\label{eq:bar:g}
\overline{g}\left(\frac{\partial}{\partial x_i},\frac{\partial}{\partial x_k}\right)=|x|^{-4}(\delta_{ik}+O(x_{n+1}^n|x|^{-2n})),\:\:i,k=1,...,n+1.
\end{equation}
On the one hand,
\begin{equation}\label{eq:g}
g_{ik}(x)=g\left(\frac{\partial}{\partial x_i},\frac{\partial}{\partial x_k}\right)=G_p(|x|^{-2}x)^{\frac{4}{n-2\gamma}}\overline{g}\left(\frac{\partial}{\partial x_i},\frac{\partial}{\partial x_k}\right)\:\:i,k=1,...,n+1.
\end{equation}
Let $b=\min\{1,2\gamma\}.$ Using the results in \cite{mayer-ndiaye1}, which are also obtained in the coordinates \eqref{eq:coord}, we have  
\begin{equation}\label{eq:G}
\begin{aligned}
	G_{p}(y)^{\frac{4}{n-2\gamma}}&=(|y|^{2\gamma-n}+A+O(|y|^{b}))^{\frac{4}{n-2\gamma}}\\
	&=|y|^{-4}\left(1+\frac{4A}{n-2\gamma}|y|^{n-2\gamma}+O(|y|^{\min\{n-2\gamma+b,2(n-2\gamma)\}})\right)\\
	&=|x|^{4}\left(1+\frac{4A}{n-2\gamma}|x|^{2\gamma-n}+O(|x|^{-{\min\{n-2\gamma+b,2(n-2\gamma)\}}})\right).
\end{aligned}
\end{equation}
It now follows from \eqref{eq:g},  \eqref{eq:bar:g} and  \eqref{eq:G} that
$$g_{ik}(x)=\left(1+\frac{4A}{n-2\gamma}|x|^{2\gamma-n}+O(|x|^{-{\min\{n-2\gamma+b,2(n-2\gamma)\}}})\right)\delta_{ik}+O(x_{n+1}^n|x|^{-2n}).$$
Similarly,
$$
	G_{p}(y)^{\frac{2}{n-2\gamma}}=|x|^{2}\left(1+\frac{2A}{n-2\gamma}|x|^{2\gamma-n}+O(|x|^{-{\min\{n-2\gamma+b,2(n-2\gamma)\}}})\right)
$$
and $\rho(y)=y_{n+1}$ gives 
$$v(x)=\left(1+\frac{2A}{n-2\gamma}|x|^{2\gamma-n}\right)x_{n+1}+x_{n+1}O(|x|^{-{\min\{n-2\gamma+b,2(n-2\gamma)\}}}).$$
This, together with the regularity property of $G_p$ obtained in \cite{mayer-ndiaye1}, proves that the asymptotic expansions \eqref{eq:can:exp} and \eqref{eq:can:exp:v}  hold with  $a=\frac{4A}{n-2\gamma}$ for some $0<\epsilon<1$. The statement about the mass holds in view of Proposition~\ref{lemma:mass}.

%It remains to prove the vanishing of $R_{(g,v,m)}$ and $\mathcal H_{(g,v,m)}$. 

By formula \eqref{eq:L:div}, the first equation of \eqref{system:green} says that $L_{(\overline g,\rho,m)}G_p=0$ on $\overline X\backslash\{p\}$. Using the conformal property \eqref{eq:conf:L} we see that $R_{(g,v,m)}=0$ on $M\backslash \Sigma$. 
As for $\mathcal H_{(g,v,m)}$, it follows from \eqref{eq:conf:B} and $v=G_p^{\frac{2}{m+n-1}}\rho$ that 
$$
\mathcal H_{(g,v,m)}^{(s)}=G_p^{\frac{m+n+1}{m+n-1}}\left(\frac{2(m+n)}{m+n-1}\partial_{\overline \eta_s} G_p+\mathcal H_{(\overline g,\rho,m)}^{(s)} G_p\right),
$$
so that
$$
s^m\mathcal H_{(g,v,m)}^{(s)}=G_p^{\frac{m-n-1}{m+n-1}}\left(\frac{2(m+n)}{m+n-1}\rho^m\partial_{\overline \eta_s} G_p+\mathcal H_{(\overline g,\rho,m)}^{(s)} \rho^m G_p\right),
$$ 
where the quantities with $s$ refer to the level sets $\Sigma_s=v^{-1}(s)$ and the unit normal vector $\overline \eta_s$ is calculated with respect to the metric $\overline g$. 

Fix a compact set $K\subset \Sigma$ and identify points on $\Sigma_s$ with corresponding ones in $\Sigma$ for small $s$. It follows from \cite{mayer-ndiaye1} (see Equations (35), (37) and (45)) that 
$$
\sup_{K}|\partial_\rho G_p|\leq C_K\rho, \qquad\text{for some }C_K>0.
$$
Since $N$ is totally geodesic with respect to $\overline g$, we can find $C'_K>0$ such that 
$$
\sup_K|H_{\overline g}^{(s)}|\leq C'_Ks.
$$
Hence, there is $C''_K>0$ satisfying  
$$
\sup_K s^m|\mathcal H_{(g,v,m)}^{(s)}|\leq C''_Ks^{m+1},
$$
whose right-hand side goes to zero because $m>-1$.
\end{proof}

We will end this section with our model example of C.C.E. manifolds. Let $$\mathbb{B}^{n+1}=\{y\in \mathbb R^{n+1}\:|\: |y|<1\}$$ be the unit ball with hyperbolic metric $$g_{\mathbb B}=\frac{4}{(1-|y|^2)^2}\delta,$$ where $\delta$ is the Euclidean metric in $\mathbb R^{n+1}$. A defining function in this case is $$\rho_{\mathbb B}=\frac{1-|y|^2}{2}.$$ Then $\overline{g_{\mathbb{B}}}=\rho_{\mathbb B}^2g_{\mathbb{B}}=\delta.$ This makes $(\mathbb{B}^{n+1}, g_{\mathbb{B}})$ a C.C.E. manifold with $(\overline{\mathbb{B}^{n+1}},\overline{g_{\mathbb{B}}})$ being exactly the Euclidean ball. Similarly let $$\mathbb{H}^{n+1}=\{x\in \mathbb R^{n+1}\:|\:x_{n+1}>0\}$$ be the half space with hyperbolic metric $$g_{\mathbb H}=\frac{1}{x_{n+1}^2}\delta.$$ Then $x_{n+1}^2g_{\mathbb{H}}=\delta.$ Now the two hyperbolic models are related by the conformal map 
$\Phi: \mathbb{B}^{n+1}\to \mathbb{R}^{n+1}_+$,

$$ x_i=\Phi_i(y)=\frac{2y_i}{|y-e_{n+1}|^2}, \:\: i=1,2,...,n,$$

$$x_{n+1}=\Phi_{n+1}(y)=\frac{1-|y|^2}{|y-e_{n+1}|^2},$$
where $e_{n+1}=(0,...,0,1)\in \mathbb R^{n+1}.$ 
It follows from direct calculations that 
$$\Phi^*\delta=\varphi^{\frac{4}{n-2\gamma}}\overline{g_{\mathbb{B}}}\qquad\text{and}\qquad x_{n+1}=\varphi^{\frac{2}{n-2\gamma}}\rho_{\mathbb{B}},$$ 
where 
$$\varphi(y)=\left(\frac{2}{|y-e_{n+1}|^2}\right)^{\frac{n-2\gamma}{2}}.$$ 
	It follows from a direct calculation that
	$$-\text{div}_{\overline{g_{\mathbb{B}}}}(\rho_{\mathbb{B}}^m\nabla_{\overline{g_{\mathbb{B}}}}\varphi)+E(\rho_{\mathbb{B}})\varphi=0, \:\: \text{in}\:\: \mathbb{B}^{n+1},$$ with 
	$$-\lim_{\rho_{\mathbb{B}}\to 0}(\rho_{\mathbb{B}}^m\partial_{\rho_{\mathbb{B}}}\varphi)=c(n,\gamma)\delta_{p}, \:\: \text{on}\:\: \partial \mathbb{B}^{n+1},$$ 
holding for some constant $c(n,\gamma)>0$ when $m>0$ for $p=e_{n+1}$. However, although $\rho_{\mathbb{B}}$ is a defining function for the hyperbolic metric ${g_{\mathbb{B}}}$ on $\mathbb{B}^{n+1}$, it is not its (unique) geodesic defining function as $|\nabla_{\overline{g_{\mathbb{B}}}}\rho_{\mathbb{B}}|=1$ holds only on $\partial \mathbb{B}^{n+1}$. As a consequence, $\varphi$ is not the Green's function of Proposition~\ref{propo:green}, Conjecture~\ref{conj:green}, and Theorem~\ref{propo:as:flat}.
In any case, this construction shows how our model, $(\mathbb R^{n+1}_+, \delta, x_{n+1}^mdx)$, can be obtained from the hyperbolic space. 

%%%%%%%%%%%%%%%%%%%%%%%%%%%%%%%%%%%%%%%%%%%%%%%%%%%%%%%%%%%%%%%%%%%%%%%%%%%%%%%%

%%%%%%%%%%%%%%%%%%%%%%%%%%%%%%%%%%%%%%%%%%%%%%%%%%%%%%%%%%%%%%%%%%%%%%%%%%%%%%%%%%%%%%%%%

\section{Total charges and static smooth metric measure spaces}\label{sec:interpr:mass}

%In this section, we justify the choice of expression \eqref{eq:mass} to define the mass $\mathfrak m_{(g,v,m)}$. We will perform formal calculations that fits Michel's general construction in \cite{michel}, to which our paper adds extra terms coming from the boundary $\Sigma$. Michel´s construction consists in identifying certain divergence terms in the linearization of the local charge density quantities, isolating the charge integrands in Definition \ref{def:mass}. Namely, the divergencies of the charge integrands $\mathbb U$ and $\overline{\mathbb U}$ appear in the linearization expansions of the local charge densities $R_{(g,v,m)}$ and $\mathcal H_{(g,v,m)}$ and the remaining parts of the expansions are the formal adjoints of those linear operators; see \eqref{eq:div:f} and \eqref{eq:div:bdry} below. The vanishing of those formal adjoints terms determines the static equations \eqref{sys:static} and \eqref{sys:static:bdry} below, which gives rise to the concept of what we call a static SMMS with boundary in Definition \ref{def:static} below.

In this section, we develop a general framework for defining total charges on smooth metric measure spaces with boundary, following Michel’s construction in \cite{michel}. Starting from the linearizations of the weighted scalar curvature $R_{(g,v,m)}$ and the modified weighted mean curvature $\mathcal H_{(g,v,m)}$, which play the role of local charge densities, we identify the corresponding formal adjoint equations and boundary conditions. Their solutions lead naturally to the notion of a static smooth metric measure space with boundary and to a space of static potentials. Each static potential of an underlying asymptotic model then gives rise to a total charge. In the particular case of the Euclidean half-space, the charge associated with the constant static potential recovers the mass $\mathfrak m_{(g,v,m)}$ defined in \eqref{eq:mass}, thereby justifying the choice of its expression.

Let $(M,g,v^mdvol_g)$ be a SMMS with boundary $\Sigma=v^{-1}(0)$, $dv\neq 0$ on $\Sigma$, $-1<m<1$. Assume that $g$ and $v$ are smooth. We begin with the linearization of the weighted scalar curvature $R_{(g,v,m)}$ with respect to $g$ and $v$. If $f$ is any smooth function on $M$, Proposition \ref{propo:R:var} in the Appendix gives
\begin{align*}
fv^mR^{\bullet}_{(g,v,m)}(e,0)=\,
&\text{div}_g\big(fv^m(\text{div}_ge-d\,\text{tr\,}_g e)+mv^{m-1}f\nabla_gv\righthalfcup e\big)
\\
&+\text{div}_g\big(-v^m \nabla_gf\righthalfcup e+v^m\text{tr\,}_g e\,df\big)
\\
&+\big\langle -f\text{Ric}_g+mfv^{-1}\nabla^2_gv-mv^{-1}dv\otimes df,e \big\rangle_gv^m
\\
&+\big\langle \nabla^2_g f-\Delta_gf\,g+mv^{-1}dv\otimes df-mv^{-1}g(\nabla_g v,\nabla_g f)g,e \big\rangle_gv^m
\end{align*}
and
\begin{align*}
fv^m&R^{\bullet}_{(g,v,m)}(0,w)=\\
&\text{div}_g\big(2mv^{m-1}(w\,df-f\,dw)\big)+2mv^{m-3}\big(fv\Delta_gv
\\
&+(m-1)f|\nabla_gv|^2_g-(m-1) f-(m-1)vg(\nabla_g v,\nabla_g f)-v^2\Delta_gf\big)w.
\end{align*}
Define the charge integrand
\begin{align}\label{eq:U:1}
\mathbb U_f(e,w)=\,
&f(\text{div}_ge-d\,\text{tr}_ge+mv^{-1} \nabla_gv\righthalfcup e)
\\
&-\nabla_gf\righthalfcup e+\text{tr}_ge\,df+2mv^{-1}(w\,df-f\,dw),\notag
\end{align}
so that
\begin{equation}\label{eq:div:f}
fv^mR^{\bullet}_{(g,v,m)}(e,w)=\text{div}_g(\mathbb U_f(e,w)v^m)+\langle R_{(g,v,m)}^{\bullet,*}f,(e,w)\rangle v^m.
\end{equation}
Here,
$R_{(g,v,m)}^{\bullet,*}f=\big((R_{(g,v,m)}^{\bullet,*}f)_1, (R_{(g,v,m)}^{\bullet,*}f)_2\big)$
is the formal adjoint, where
$$
(R_{(g,v,m)}^{\bullet,*}f)_1=-f\text{Ric}_g+mfv^{-1}\nabla^2_gv
+\nabla^2_g f-\Delta_gf\,g-mv^{-1}g(\nabla_g v,\nabla_g f)g,
$$
and 
\begin{align*}
(R_{(g,v,m)}^{\bullet,*}f)_2=
2m\big(fv^{-2}\Delta_gv&+(m-1)v^{-3}f|\nabla_gv|^2_g-(m-1) v^{-3}f
\\
&-(m-1)v^{-2}g(\nabla_g v,\nabla_g f)-v^{-1}\Delta_gf\big).
\end{align*}
Thus, the linearization of the weighted scalar curvature decomposes into a divergence term, which will give rise to the corresponding total charge, and a term involving the formal adjoint of the linearized operator. Following the general principle in \cite{michel}, this suggests considering functions $f$ lying in the kernel of the formal adjoint.
The vanishing of the formal adjoint applied to $f$, namely $R_{(g,v,m)}^{\bullet,*}f=(0,0)$,  determines the static equations    
\begin{align}\label{sys:static}
\begin{cases}
\nabla_g^2f-f\text{Ric}^m_{g,v}-\Delta^m_{g,v}f\,g=0,
\\
2mv^{-1}(-\Delta^{m-1}_{g,v}f+fv^{-1}\Delta^{m-1}_{g,v}v-(m-1) fv^{-2})=0,
\end{cases}
\end{align}
in the interior $M\backslash \Sigma$, where
$$
\text{Ric}^k_{g,v}=\text{Ric}_g-kv^{-1}\nabla^2_gv
\:\:\:\:
\text{and}
\:\:\:\:
\Delta^k_{g,v}h=\Delta_gh+kv^{-1}g(\nabla_g v, \nabla_g h),
\:\:\:\:k\in\mathbb R.
$$

We next perform the analogous calculation for the modified weighted mean curvature. This will provide the boundary counterpart of \eqref{eq:div:f} and determine the boundary conditions that must accompany the interior static equations \eqref{sys:static}.
We calculate the linearization of the modified weighted mean curvature $\mathcal H^{\bullet}_{(g,v,m)}$ of the level sets $\Sigma_s=v^{-1}(s)\subset M$, for small $s>0$.
Using Proposition \ref{propo:H:var} in the Appendix, we have
\begin{align*}
		2fv^m\mathcal H^{\bullet}_{(g,v,m)}
		%&=fv^m(d\text{tr}_ge-\text{div}_ge)(\eta)-f\text{div}_{g|_\Sigma}(v^m \eta\righthalfcup e)
		%\\
		%&\hspace{0.6cm}-\langle f\pi_g,e\rangle_{g|_\Sigma}v^m-mfv^{m-1}(\nabla_gv\righthalfcup e)(\eta)
		%\\
		%&\hspace{0.6cm}+2mfv^{m-1}dw (\eta)
		%\\
		&=fv^m(d\text{tr}_ge-\text{div}_ge)(\eta_s)
		\\
		&\hspace{0.6cm}-\text{div}_{g|_{\Sigma_s}}(fv^m \eta_s\righthalfcup e)+v^m\nabla_{g|_{\Sigma_s}}f\righthalfcup e (\eta_s)
		\\
		&\hspace{0.6cm}-\langle f\pi^{(s)}_g,e\rangle_{g|_{\Sigma_s}}v^m-mfv^{m-1}(\nabla_gv\righthalfcup e)(\eta_s)
		\\
		&\hspace{0.6cm}+2mfv^{m-1}dw (\eta_s)-2mv^{m-2}w(dv(\eta_s)+1)f
		\\
		&=-v^m\mathbb U_f(e,w)(\eta_s)+\text{div}_{g|_{\Sigma_s}}(v^m\overline{\mathbb U}_f(e,w))
		-e(\eta_s,\eta_s)df(\eta_s)v^m
		\\
		&\hspace{0.6cm}-\langle f\pi^{(s)}_g,e\rangle_{g|_{\Sigma_s}}v^m+v^m(\text{tr}_ge +2mv^{-1}w)df(\eta_s)
		\\
		&\hspace{0.6cm}-2mv^{m-2}w(dv(\eta_s)+1)f,
\end{align*}
where $\overline{\mathbb U}_f(e,w)=-f(\eta_s\righthalfcup e)|_{\Sigma_s}$ is the boundary charge integrand and $\pi_g^{(s)}$ is the second fundamental form of $\Sigma_s$. Then
\begin{equation}\label{eq:div:bdry}
2fv^m\mathcal H^{\bullet}_{(g,v,m)}=-v^m\mathbb U_f(e,w)(\eta_s)+\text{div}_{g|_{\Sigma_s}}(v^m\overline{\mathbb U}_f(e,w))+\langle \mathcal H_{(g,v,m)}^{\bullet,*}f, (e,w)\rangle_{\Sigma_s} v^m,
\end{equation}
where $\mathcal H_{(g,v,m)}^{\bullet,*}f=\big((\mathcal H_{(g,v,m)}^{\bullet,*}f)_1, (\mathcal H_{(g,v,m)}^{\bullet,*}f)_2\big)$ with
$$
(\mathcal H_{(g,v,m)}^{\bullet,*}f)_1=-f\pi^{(s)}_g+df(\eta_s)g|_{\Sigma_s},
\qquad\text{and}\qquad
(\mathcal H_{(g,v,m)}^{\bullet,*}f)_2=2mv^{-1}df(\eta_s)-2mv^{-2}f(dv(\eta_s)+1).
$$
%Here, $\widetilde \pi_g$ is the extension to the tangent space of $M$ of the second fundamental form $\pi_g$ of $\Sigma_s$ in such a way that  $\eta_s\righthalfcup\widetilde \pi_g=0$. 
Together, \eqref{eq:div:f} and \eqref{eq:div:bdry} constitute the interior and boundary Green-type identities underlying the total charge construction.
In view of \eqref{eq:can:exp:v} it is natural to assume that $w$ vanishes along $\Sigma$. In this case, we write $w=v\omega$ near $\Sigma$ for some smooth function $\omega$. Then on $\Sigma_s=v^{-1}(s)$ one has $w|_{\Sigma_s}=s\omega|_{\Sigma_s}$. This leads to the static boundary conditions 
\begin{align}\label{sys:static:bdry}
\lim_{s\to 0^+} s^{m}(-f\pi^{(s)}_g+df(\eta_s){g|_{\Sigma_s}})=0
\qquad\text{and}\qquad
\lim_{s\to 0^+} ms^{m} (df(\eta_s)-s^{-1}f(dv(\eta_s)+1))=0,
\end{align}
uniformly on compact sets.
Thus, equations \eqref{sys:static} and \eqref{sys:static:bdry} are precisely the interior and boundary conditions that eliminate the non-divergence terms in the linearized identities  \eqref{eq:div:f} and \eqref{eq:div:bdry}.
\begin{definition}\label{def:static}
We say that $(M,g,v^mdvol_g)$ is a {\rm{static SMMS with boundary}} if the system consisting of equations \eqref{sys:static} and \eqref{sys:static:bdry} admits a non-trivial solution $f$. Any of these non-trivial solutions is said to be a {\rm{static potential}} and  the vector space of all solutions to these equations is called the {\rm{space of static potentials}}. 
\end{definition}
\begin{remark}
When $m=0$ the equations \eqref{sys:static} and \eqref{sys:static:bdry} become the static equation with boundary considered in \cite[Equation (2.3)]{cruz-vitorio}, \cite[Remark 2.2]{almaraz-lima1}, \cite[Equation (1.3)]{cruz-nunes}, \cite[Equation (1.1)]{cruz-santos} and \cite[Equation (2.5)]{ho-huang}. Hence, Definition \ref{def:static} can be viewed as an extension of the concept of static manifolds with boundary \cite[Definition 1.2]{almaraz-lima2} to SMMS with non-vanishing parameter $m$. Static manifolds with boundary are rigid with respect to compact deformations preserving the dominant energy conditions; see \cite{almaraz-lima2} (we refer the reader to \cite{medvedev1, medvedev2} for other rigidity and uniqueness results on static manifolds with boundary). On the other hand, non-static manifolds with boundary can be locally deformed preserving the dominant energy conditions; see \cite{sheng}. Conjecture \ref{our:corol} is the rigidity result for SMMS asymptotically modelled on $(\mathbb R^{n+1}_+, \delta, x_{n+1}^mdx)$ corresponding to the one in \cite{almaraz-lima2}.
\end{remark}
\begin{example}
The Euclidean half-space $(\mathbb R^{n+1}_+, \delta, x_{n+1}^mdx)$ is a static SMMS with  space of static potentials generated by $\{1,x_1,...,x_{n}\}$. In this case, the equations \eqref{sys:static:bdry} are satisfied identically, not only asymptotically. This is our model example and will be discussed below.
\end{example}
\begin{example}
Consider the hyperbolic half-space $\mathbb H^{n+1}_+=\{x\in \mathbb R^{1,n+1}\:;\:-x_0^2+x_1^2+...+x_{n+1}^2=-1, \:x_{n+1}\geq 0 \}$ in the hyperboloid model, where $ \mathbb R^{1,n+1}$ is the Minkowski space. 
Then  $(\mathbb H^{n+1}_+, g_{\mathbb H^{n+1}_+}, x_{n+1}^mdvol_{\mathbb H^{n+1}_+})$ is a static SMMS having the space of static potentials generated by $\{x_0|_{\mathbb H^{n+1}_+},...,x_{n}|_{\mathbb H^{n+1}_+}\}$. 
The equations  \eqref{sys:static} and \eqref{sys:static:bdry} can be directly verified in view of \cite[Proposition 1.6]{almaraz-lima2}.
\end{example}
\begin{example}
If $\mathbb S^{n+1}_+=\{x\in \mathbb R^{n+2}\:;\:x_1^2+...+x_{n+2}^2=1, \:x_{n+2}\geq 0 \}$ is the round hemisphere, then $(\mathbb S^{n+1}_+, g_{\mathbb S^{n+1}_+}, x_{n+2}^mdvol_{\mathbb S^{n+1}_+})$ is a static SMMS with  space of static potentials generated by $\{x_1|_{\mathbb S^{n+1}_+},...,x_{n+1}|_{\mathbb S^{n+1}_+}\}$. 
\end{example}

Now we fix an underlying background static SMMS $(M_0, g_0, v_0^m dvol_{g_0})$ with boundary $v_0^{-1}(0)=\Sigma_0$ as in Definition \ref{def:static}.
Let $\mathcal N_0$ denote the space of static potentials in the underlying SMMS.  Precisely, $\mathcal N_0$  is the set of smooth functions $f$ such that  \eqref{sys:static} and \eqref{sys:static:bdry} hold in $(M_0, g_0, v_0^mdvol_{g_0})$.
Let $\{M_0^{(\rho)}\}_{\rho>0}$ be an exhaustion of $M_0$ by compact sets $M_0^{(\rho)}$ with (non-smooth) boundaries $\partial M_0^{(\rho)}=S^{(\rho)}\cup \Sigma_0^{(\rho)}$ where

$\bullet$ $\Sigma_0^{(\rho)}\subset \Sigma_0$ is a compact set with boundary  $\partial \Sigma_0^{(\rho)}$ such that $\{\Sigma_0^{(\rho)}\}_{\rho>0}$ is an exhaustion of $\Sigma_0$;

$\bullet$  $S^{(\rho)}\subset M_0$ is a smooth hypersurface with boundary $\partial S^{(\rho)}=\partial \Sigma_0^{(\rho)}$.

Let $(M,g,v^mdvol_g)$ be a SMMS with boundary $v^{-1}(0)=\Sigma$, which we assume to be modelled by $(M_0, g_0, v_0^mdvol_{g_0})$ in the following sense: there exist compact sets $K_0\subset M_0$ and $K\subset M$, and a diffeomorphism 
$\phi:M\backslash K\to M_0\backslash K_0$. Set
$$
e:=\phi_*g-g_0
\qquad{and}\qquad
w:=v\circ \phi^{-1}-v_0.
$$
Since $w|_\Sigma=0$, we write $w=v_0\omega$.
For small $s>0$, define $M^{(\rho)}_{0,s}=M^{(\rho)}_0 \cap \{v_0\geq s\},$ ${\Sigma}^{(\rho)}_{0,s}=M^{(\rho)}_0 \cap \{v_0=s\},$ and $S^{(\rho)}_{s}=S^{(\rho)} \cap \{v_0\geq s\}.$ 
For each $f\in\mathcal N_0$,  assume that the following limit exists: 
\begin{align}\label{limit}
\lim_{\rho\to\infty}\lim\limits_{s\to 0^+}\left\{\int_{M_{0,s}^{(\rho)}}f R^{\bullet}_{(g_0,v_0,m)}(e,w)v_0^mdvol_{g_0}+\int_{\Sigma_{0,s}^{(\rho)}}2f \mathcal H^{\bullet}_{(g_0,v_0,m)}(e,w)v_0^md\Sigma_s\right\}.
\end{align}
Observe that, for each $\rho,s>0$,
\begin{align}\label{int:parts}
\int_{S_s^{(\rho)}}&\langle \mathbb U_f(e,w), \mu_{g_0}\rangle v_0^m d\sigma_{g_0}+\int_{\partial S_s^{(\rho)}}\langle \overline{\mathbb U}_f(e,w), \nu_{g_0}\rangle v_0^m d\overline \sigma_{g_0}
\\
&=\int_{M_{0,s}^{(\rho)}}f R^{\bullet}_{(g_0,v_0,m)}(e,w)v_0^mdvol_{g_0}+\int_{\Sigma_{0,s}^{(\rho)}}2f \mathcal H^{\bullet}_{(g_0,v_0,m)}(e,w)v_0^md\Sigma_s\notag
\\
&\hspace{0.5cm}-\int_{\Sigma_{0,s}^{(\rho)}}\langle -f\widetilde \pi_{g_0}+df(\eta_s)g_0,e\rangle v_0^md\Sigma_s
-2m\int_{\Sigma_{0,s}^{(\rho)}}(df(\eta_s)-v_0^{-1}f(dv_0(\eta_s)+1))\omega v_0^md\Sigma_s\notag
\end{align}
hold by an integration by parts, where $\mu_{g_0}$ and $\nu_{g_0}$ are the unit outward pointing normal and co-normal vectors to the hypersurfaces $S_s^{(\rho)}$ and $\partial S_s^{(\rho)}$ respectively, and $d\sigma_{g_0}$ and $d\overline\sigma_{g_0}$ are the area elements. The last two integrals in \eqref{int:parts} vanish as $s\to 0^+$ due to \eqref{sys:static:bdry}.
The linear functional  
$$
\mathcal Q_{(g,v,m)}:\mathcal N_0\to\mathbb R
$$
defined by 
\begin{equation}\label{limit:2}
\mathcal Q_{(g,v,m)}(f)=\lim_{\rho\to\infty}\lim\limits_{s\to 0^+}\left\{
\int_{S_s^{(\rho)}}\langle \mathbb U_f(e,w), \mu_{g_0}\rangle v_0^md\sigma_{g_0}
+\int_{\partial S_s^{(\rho)}}\langle \overline{\mathbb U}_f(e,w), \nu_{g_0}\rangle v_0^md\overline \sigma_{g_0}
\right\}.
\end{equation}
is called the {\it{total charge functional}} of $(M,g,v^mdvol_g)$ relative to the static background $(M_0, g_0,v_0^mdvol_{g_0})$.
This is similar to the total charge in \cite[page 052504-4]{michel}.

In order to obtain the expression \eqref{eq:mass}, we particularize our analysis to the case the static SMMS with boundary $(M_0, g_0,v_0^mdvol_{g_0})$ is $(\mathbb R^{n+1}_+, \delta, x_{n+1}^mdx)$. We have
$$
\text{Ric}^k_{\delta,x_{n+1}}=\text{Ric}_{\delta}-kx_{n+1}^{-1}\nabla_\delta^2 x_{n+1}=0
$$
and
$$
\Delta^k_{\delta,x_{n+1}}h=\Delta_{\delta}h+kx_{n+1}^{-1}\delta(\nabla_\delta x_{n+1}, \nabla_\delta h)=\Delta_\delta h+kx_{n+1}^{-1}\frac{\partial h}{\partial x_{n+1}}.
$$
So for $m\neq 0$, the static system (\ref{sys:static}) becomes
\begin{align*}%\label{eq:f:4}
\begin{cases}
\nabla_\delta^2f-(\Delta_\delta f+mx_{n+1}^{-1}\displaystyle\frac{\partial f}{\partial x_{n+1}})\delta=0,
\\
-\Delta_\delta f-(m-1)x_{n+1}^{-1}\displaystyle\frac{\partial f}{\partial x_{n+1}}=0,
\end{cases}
\end{align*}
The first equation is rewritten as 
$$
\nabla_\delta^2f=x_{n+1}^{-1}\frac{\partial f}{\partial x_{n+1}}\delta 
$$
which yields $\Delta_\delta f=(n+1)x_{n+1}^{-1}\frac{\partial f}{\partial x_{n+1}}$ after tracing.
Combining this with the second static equation gives $(n+m)x_{n+1}^{-1}\frac{\partial f}{\partial x_{n+1}}=0$. Since $m\neq -n$, the static system finally becomes
\begin{align}\label{eq:f:3}
	\nabla_\delta^2f=0
	\:\:\:\:\text{and}\:\:\:
	x_{n+1}^{-1}\frac{\partial f}{\partial x_{n+1}}=0\,,
\end{align}
Observe that (\ref{eq:f:3}) is defined for $x_{n+1}>0$ and does not depend on $m$. 
On the other hand, the boundary static equations \eqref{sys:static:bdry} hold whenever
\begin{align}\label{eq:f:4}
	\frac{\partial f}{\partial x_{n+1}}=0\,,\:\:\text{on}\:\:x_{n+1}\geq  0.
\end{align}
One checks by the same arguments that when $m=0$
$$\nabla_\delta^2f=0,$$ and 
$$\frac{\partial f}{\partial x_{n+1}}=0\,,\:\:\text{on}\:\:x_{n+1}\geq  0.$$
It is readily seen that the space of all solutions on $\mathbb R^{n+1}_+$ to the equations \eqref{eq:f:3} and \eqref{eq:f:4} is generated by  $\{1, x_1, ..., x_n\}$. Indeed, this follows from the  fact that $\{1, x_1, ..., x_{n+1}\}$ generates the space of functions $\widetilde f$ satisfying $\nabla_\delta^2\widetilde f=0$ in $\mathbb R^{n+1}$. In analogy with the classical ADM mass and the mass in \cite{almaraz-barbosa-lima}, we are led to choose $f=1$ so that the limit \eqref{limit:2} becomes 
\begin{align*}%\label{limit:2}
\mathcal Q_{(g,v,m)}(1)=\lim_{\rho\to\infty}\lim\limits_{s\to 0^+}&\Big\{
\int_{|x|=\rho,\,x_{n+1}>s}\langle \mathbb U_1(e,w), \mu_{\delta}\rangle x_{n+1}^md\sigma
\\
&+\int_{|x|=\rho, \,x_{n+1}=s}\langle \overline{\mathbb U}_1(e,w), \nu_{\delta}\rangle x_{n+1}^md\overline \sigma
\Big\}.\notag
\end{align*}
This expression coincides with the one in \eqref{eq:mass} if we make $\mathfrak m_{(g,v,m)}=\mathcal Q_{(g,v,m)}(1)$, $\mathbb U=\mathbb U_1$ and $\overline{\mathbb U}=\overline{\mathbb U}_1$.
This justifies Definition \ref{def:mass}. 

\begin{remark}
As the total charge should be regarded as an element of the dual space $\mathcal Q_{(g,v,m)}\in\mathcal N_0^*$, the use of coordinate static potentials also suggests 
$$
(\mathcal Q_{(g,v,m)}(x_1),...,\mathcal Q_{(g,v,m)}(x_n))
$$ 
as a candidate of a center of mass for the ASSMMS $(M,g,v^mdvol_g)$.
%by means of $\mathbb U_{x_1}$, ..., $\mathbb U_{x_n}$ and $\overline{\mathbb U}_{x_1}$, ..., $\overline{\mathbb U}_{x_n}$
In the case $m=0$, this center of mass already appears in \cite{almaraz-lima} and in \cite{lima-girao-montalban}.
Furthermore, we expect that other backgrounds, such as the hyperbolic model, should give their corresponding mass-type charges. 
\end{remark}

%Set $\mathbb U(e,\xi)=\mathbb U_1(e,\xi)$.
%Observe that 
%\begin{align*}
%\ged{
%	\mathbb U(e,w)=\,
%	\text{div}\,e-d\,\text{tr\,}e+my^{-1} dy\righthalfcup e-2my^{-1}\,dw,
%}
%\end{align*}
%or, in coordinates,
%\begin{align*}\marginnote{Os indices $i,\, j$ variam de $1$ a $n+1$,  mas uso o abuso de notacao $n+1=y$.}
%\mathbb U_i(e,w)=\,
%e_{ij,j}-e_{jj,i}+my^{-1} e_{iy}-2my^{-1}w_{,i}\,.
%\end{align*}

%%%%%%%%%%%%%%%%%%%%%%%%%%%%%%%%%%%%%%%%%%%%%%%%%%%%%%%%%%%%%%%%%%%%%%%%%%%%%%%

\section{The mass as a geometric invariant}\label{sec:geom:inv}

%\textcolor{red}{Give credits for this type of construction. Maybe Chrusciel or Bartnik.}

In this section, we prove Theorem \ref{thm:geom:inv}. The case $m=0$ is proved in \cite{almaraz-barbosa-lima}. We first assume that $0<m<1$.
We will make use of the notation
$$
\mathbb U^{(h,f)}(e,w)=\text{div}_h e-d\text{tr}_h e+mf^{-1}\nabla_h f\righthalfcup e-2mf^{-1} dw
$$
where $h$ is a metric, $f$ is a positive smooth function, $e$ is a symmetric two-tensor and $w$ is a smooth function.

Fix an asymptotic coordinate system $(x_1,...,x_{n+1})$ coming from a diffeomorphism $\phi$ as in Definition \ref{def:as:flat}. Setting $e=\phi_*g-\delta$ and $w=v\circ \phi^{-1}-x_{n+1}$, the equation \eqref{eq:mass} is written as 
\begin{equation}\label{eq:geom:inv:1}
	\mathfrak m_{(g,v,m)}=\lim_{\rho\to\infty}\int_{|x|=\rho,\:x_{n+1}>0}\sum_{i=1}^{n+1}\mathbb U_i^{(\delta,x_{n+1})}(e,w)\frac{x_i}{|x|}x_{n+1}^m d\sigma,
\end{equation} 
since the integral limit on $x_{n+1}=0$ vanishes for $m>0$. We have the expansions
$$
e_{ij}(x)=a|x|^{1-n-m}\delta_{ij}+O'(|x|^{\epsilon-n-m}),\qquad{i,j=1,...,n+1},
$$
and
$$
w(x)=\frac{a}{2}|x|^{1-n-m}x_{n+1}+x_{n+1}O'(|x|^{\epsilon-n-m}),
$$
for some $0<\epsilon<1$.

Consider another asymptotic coordinate system $(y_1,..., y_{n+1})$ coming from a diffeomorphism $\widetilde\phi$. We want to prove that the right-hand side of \eqref{eq:geom:inv:1} equals
$$
\lim_{\rho\to\infty}\int_{|y|=\rho,\:y_{n+1}>0}\sum_{i=1}^{n+1}\mathbb U_i^{(\delta, y_{n+1})}(\widetilde e, \widetilde w)\frac{y_i}{|y|}y_{n+1}^m d\sigma,
$$
where $\widetilde e=\widetilde \phi_*g-\delta$ and $\widetilde w=v\circ \widetilde\phi^{-1}-y_{n+1}$.
We assume 
$$
\widetilde e_{ij}(y)=\widetilde a|y|^{1-n-m}\delta_{ij}+O'(|y|^{\epsilon-n-m}),\qquad{i,j=1,...,n+1},
$$
and
$$
\widetilde w(y)=\frac{\widetilde a}{2}|y|^{1-n-m}y_{n+1}+y_{n+1}O'(|y|^{\epsilon-n-m}).
$$
In particular we can choose $(n+m-1)/2<\tau<n-1$ such that $e_{ij}(x)=O(|x|^{-\tau})$, $w(x)=O(x_{n+1}|x|^{-\tau})$, $\widetilde e_{ij}(y)=O(|y|^{-\tau})$ and $\widetilde w(y)=O(y_{n+1}|y|^{-\tau})$. This can be done because $0<m<1$ and $n\geq 2$.

Set $\Phi=\widetilde \phi\circ\phi^{-1}$ so that $y_i(x)=\Phi_i(x)$, $i=1,...,n+1$.
It follows from Propositions 3.8 and 3.9 in \cite{almaraz-barbosa-lima} that, after composing $\Phi$ with an isometry of $\mathbb{R}^{n+1}_+$, we can assume
\begin{equation}\label{eq:geom:inv:2}
	\sum_{i=1}^{n+1}|x_i-y_i|=O'(|x|^{1-\tau}),
\end{equation}
so that the vector field $Z_i=y_i-x_i$ on $\mathbb R^{n+1}_+$ satisfies 
\begin{equation}\label{eq:geom:inv:4}
	\Phi^*\delta-\delta=L_Z\delta+O(|x|^{-2\tau}).
\end{equation}
Here, $L_Z$ stands for the Lie derivative. %Observe that 
%\begin{equation}\label{eq:geom:inv:5}
%\Phi_{n+1}(x)-x_{n+1}=Z_{n+1}.
%\end{equation}

%In particular, $Z_{n+1}=O(|x|^{1-\tau})$, because the left-hand side is $y_{n+1}-x_{n+1}$, but this fact will not be necessary in this proof.

\vspace{0.2cm}
\noindent
{\bf{Claim 1.}} We have
$$
\mathbb U_i^{(\delta, x_{n+1})}(L_Z\delta, L_Zx_{n+1})x_{n+1}^m=\sum_{j=1}^{n+1}\left((Z_{i,j}-Z_{j,i})x_{n+1}^m \right)_{,j},
\:\: i=1,...,n+1.
$$

Indeed, $(L_Z\delta)_{ij}=Z_{i,j}+Z_{j,i}$ and $L_Z x_{n+1}=Z(x_{n+1})=Z_{n+1}$. So,
\begin{align*}
	\mathbb U_i&(L_Z\delta,Z_{n+1})
	\\
	&=\sum_{j=1}^{n+1}(Z_{i,jj}+Z_{j,ij}-2Z_{j,ji})+mx_{n+1}^{-1}(Z_{i,n+1}+Z_{n+1,i})-2mx_{n+1}^{-1}Z_{n+1,i}
	\\
	&=\sum_{j=1}^{n+1}\left((Z_{i,j}-Z_{j,i})_{,j}+mx_{n+1}^{-1}\delta_{j\,n+1}(Z_{i,j}-Z_{j,i})\right)
\end{align*}
and Claim 1 follows from a direct computation.

\vspace{0.2cm}
\noindent
{\bf{Claim 2.}} We have 
\begin{align*}
	\lim_{\rho\to\infty}&\int_{|x|=\rho,\:x_{n+1}>0}\sum_{i=1}^{n+1}\mathbb U_i^{(\delta, x_{n+1})}(e, w)\frac{x_i}{|x|}x_{n+1}^md\sigma
	\\
	&=\lim_{\rho\to\infty}\int_{|x|=\rho,\:x_{n+1}>0}\sum_{i=1}^{n+1}\mathbb U_i^{(\delta, x_{n+1})}(\Phi^*\widetilde e, \widetilde w\circ\Phi)\frac{x_i}{|x|}x_{n+1}^md\sigma.
\end{align*}

Indeed, it follows from \eqref{eq:geom:inv:4}  that
$$
e-\Phi^*\widetilde e=\Phi^*\delta-\delta=L_Z\delta+O(|x|^{-2\tau})
$$ 
and 
$$
w-\widetilde w\circ\Phi=\Phi_{n+1}(x)-x_{n+1}=Z_{n+1}=L_Z x_{n+1}.
$$
Then it follows from Claim 1 and an integration by parts that
\begin{align*} 
	\int_{|x|=\rho,\:x_{n+1}>0}&\sum_{i=1}^{n+1}\mathbb U_i^{(\delta, x_{n+1})}(e, w)\frac{x_i}{|x|}x_{n+1}^md\sigma
	\\
	&\hspace{0.3cm}-\int_{|x|=\rho,\:x_{n+1}>0}\sum_{i=1}^{n+1}\mathbb U_i^{(\delta, x_{n+1})}(\Phi^*\widetilde e, \widetilde w\circ\Phi)\frac{x_i}{|x|}x_{n+1}^md\sigma
	\\
	&=\int_{|x|=\rho,\:x_{n+1}>0}\sum_{i,j=1}^{n+1}\left((Z_{i,j}-Z_{j,i})x_{n+1}^m \right)_{,j}\frac{x_i}{|x|}d\sigma+O(\rho^{n+m-1-2\tau})
	\\
	&=-\int_{|x|=\rho,\:x_{n+1}=0}\sum_{\alpha=1}^{n}(Z_{\alpha, n+1}-Z_{n+1,\alpha})x_{n+1}^m \frac{x_\alpha}{|x|}d\overline \sigma+O(\rho^{n+m-1-2\tau}).
\end{align*}
Here the integration by parts is performed on the truncated spherical cap $\{|x|=\rho,\ x_{n+1}>s\}$ and then lets $s\to0^+$. The last integral vanishes because $m>0$. Claim 2 holds by taking the limit as $\rho\to\infty$.

Let us use $\mu_h$ to denote the outward pointing unit normal vector field of a given hypersurface, calculated with respect to the metric $h$.
We have
\begin{align*}
	\int_{|x|=\rho,\:x_{n+1}>0}&\sum_{i=1}^{n+1}\mathbb U_i^{(\delta, x_{n+1})}(\Phi^*\widetilde e, \widetilde w\circ\Phi)\frac{x_i}{|x|}x_{n+1}^md\sigma
	\\
	&=\int_{|x|=\rho,\:x_{n+1}>0}\langle\mathbb U^{(\delta, x_{n+1})}(\Phi^*\widetilde e, \widetilde w\circ\Phi), \mu_\delta\rangle x_{n+1}^md\sigma
\end{align*} 
so that
\begin{align}\label{eq:geom:inv:6}
	\lim_{\rho\to\infty}\int_{|x|=\rho,\:x_{n+1}>0}&\sum_{i=1}^{n+1}\mathbb U_i^{(\delta, x_{n+1})}(\Phi^*\widetilde e, \widetilde w\circ\Phi)\frac{x_i}{|x|}x_{n+1}^md\sigma
	\\
	&=\lim_{\rho\to\infty}\int_{|x|=\rho,\:x_{n+1}>0}\langle\mathbb U^{(\delta, x_{n+1})}(\Phi^*\widetilde e, \widetilde w\circ\Phi), \mu_\delta\rangle x_{n+1}^md\sigma\notag
	\\
	&=\lim_{\rho\to\infty}\int_{|x|=\rho,\:x_{n+1}>0}\langle\mathbb U^{(\Phi^*\delta, \Phi_{n+1}(x))}(\Phi^*\widetilde e, \widetilde w\circ\Phi), \mu_{\Phi^*\delta}\rangle (\Phi_{n+1}(x))^md\sigma\notag
\end{align} 
as $\Phi^*\delta-\delta=O(\rho^{-\tau})$ and $x_{n+1}-\Phi_{n+1}(x)=O(x_{n+1}\rho^{-\tau})$.
Expressing the last integral in the coordinates $(y_1,...,y_{n+1})$, we have
\begin{align*}
	\lim_{\rho\to\infty}&\int_{|x|=\rho,\:x_{n+1}>0}\langle\mathbb U^{(\Phi^*\delta, \Phi_{n+1}(x))}(\Phi^*\widetilde e, \widetilde w\circ\Phi), \mu_{\Phi^*\delta}\rangle (\Phi_{n+1}(x))^md\sigma
	\\
	&=\lim_{\rho\to\infty}\int_{|\Phi^{-1}(y)|=\rho,\:y_{n+1}>0}\langle\mathbb U^{(\delta, y_{n+1})}(\widetilde e, \widetilde w), \mu_{\delta}\rangle y_{n+1}^md\sigma,
\end{align*} 
so, using Claim 2 and \eqref{eq:geom:inv:6}, we obtain
\begin{align*}
	\lim_{\rho\to\infty}&\int_{|x|=\rho,\:x_{n+1}>0}\sum_{i=1}^{n+1}\mathbb U_i^{(\delta, x_{n+1})}(e, w)\frac{x_i}{|x|}x_{n+1}^md\sigma
	\\
	&=\lim_{\rho\to\infty}\int_{|x|=\rho,\:x_{n+1}>0}\sum_{i=1}^{n+1}\mathbb U_i^{(\delta, x_{n+1})}(\Phi^*\widetilde e, \widetilde w\circ\Phi)\frac{x_i}{|x|}x_{n+1}^md\sigma
	\\
	&=\lim_{\rho\to\infty}\int_{|\Phi^{-1}(y)|=\rho,\:y_{n+1}>0}\langle\mathbb U^{(\delta, y_{n+1})}(\widetilde e, \widetilde w), \mu_{\delta}\rangle y_{n+1}^md\sigma\,.
\end{align*}
Since 
\begin{align*}
	\lim_{\rho\to\infty}&\int_{|y|=\rho,\:y_{n+1}>0}\langle\mathbb U^{(\delta, y_{n+1})}(\widetilde e, \widetilde w), \mu_{\delta}\rangle y_{n+1}^md\sigma
	\\
	&=\lim_{\rho\to\infty}\int_{|y|=\rho,\:y_{n+1}>0}\sum_{i=1}^{n+1}\mathbb U_i^{(\delta, y_{n+1})}(\widetilde e, \widetilde w)\frac{y_i}{|y|}y_{n+1}^md\sigma,
\end{align*}
it only remains to prove that 
\begin{align}\label{eq:geom:inv:7}
	\lim_{\rho\to\infty}&\int_{|y|=\rho,\:y_{n+1}>0}\langle\mathbb U^{(\delta, y_{n+1})}(\widetilde e, \widetilde w), \mu_{\delta}\rangle y_{n+1}^md\sigma
	\\
	&=\lim_{\rho\to\infty}\int_{|\Phi^{-1}(y)|=\rho,\:y_{n+1}>0}\langle\mathbb U^{(\delta, y_{n+1})}(\widetilde e, \widetilde w), \mu_{\delta}\rangle y_{n+1}^md\sigma.\notag
\end{align}
Write $\widetilde e=\widetilde e^0+O'(|y|^{\epsilon-n-m})$ and $\widetilde w=\widetilde w^0+y_{n+1}O'(|y|^{\epsilon-n-m})$, where 
$$
\widetilde e^0=\widetilde a|y|^{1-n-m}\delta
\qquad\text{and}\qquad
\widetilde w^0=\frac{\widetilde a}{2}|y|^{1-n-m}y_{n+1}.
$$
Then, equation \eqref{eq:geom:inv:7} is equivalent to 
\begin{align}\label{eq:geom:inv:8}
	\lim_{\rho\to\infty}&\int_{|y|=\rho,\:y_{n+1}>0}\langle\mathbb U^{(\delta, y_{n+1})}(\widetilde e^0, \widetilde w^0), \mu_{\delta}\rangle y_{n+1}^md\sigma
	\\
	&=\lim_{\rho\to\infty}\int_{|\Phi^{-1}(y)|=\rho,\:y_{n+1}>0}\langle\mathbb U^{(\delta, y_{n+1})}(\widetilde e^0, \widetilde w^0), \mu_{\delta}\rangle y_{n+1}^md\sigma.\notag
\end{align}

\vspace{0.2cm}
\noindent
{\bf{Claim 3.}} We have $\mathbb U_{n+1}^{(\delta,y_{n+1})}(\widetilde e^0, \widetilde w^0)y_{n+1}^m=(n+m)(m+n-1)\widetilde a|y|^{-1-n-m}y_{n+1}^{m+1}$.

Indeed, 
\begin{align*}
	\mathbb U_{n+1}^{(\delta,y_{n+1})}(\widetilde e^0, \widetilde w^0)y_{n+1}^m
	&=\sum_{j=1}^{n+1}(\widetilde e^0_{n+1\,j,j}-\widetilde e^0_{jj,n+1})y_{n+1}^m+my_{n+1}^{m-1}(\widetilde e^0_{n+1\,n+1}-2\widetilde w^0_{,n+1}).
\end{align*}
So the claim follows by observing that  
\begin{align*}
	\sum_{j=1}^{n+1}(\widetilde e^0_{n+1\,j,j}-\widetilde e^0_{jj,n+1})
	&=\sum_{j=1}^{n+1}\widetilde a(|y|^{1-n-m})_{,j}\delta_{n+1\,j}-\sum_{j=1}^{n+1}\widetilde a(|y|^{1-n-m})_{,n+1}\delta_{jj}
	\\
	&=n(m+n-1)\widetilde a|y|^{-1-n-m}y_{n+1}\end{align*}
and
\begin{align*}
	\widetilde e^0_{n+1\,n+1}-2\widetilde w^0_{,n+1}
	&=\widetilde a|y|^{1-n-m}\delta_{n+1\,n+1}-\widetilde a(|y|^{1-n-m}y_{n+1})_{,n+1}
	\\
	&=(m+n-1)\widetilde a|y|^{-1-n-m}y_{n+1}^2.
\end{align*}

In order to prove equation \eqref{eq:geom:inv:8}, we integrate by parts using Claim 3 to prove that the flux through $y_{n+1}=0$ vanishes and observe that there is $C>0$ satisfying
\begin{align*}
	\Big|\int_{|y|=\rho,\:y_{n+1}>0}&\langle\mathbb U^{(\delta, y_{n+1})}(\widetilde e^0, \widetilde w^0), \mu_{\delta}\rangle y_{n+1}^md\sigma
	\\
	&-\int_{|\Phi^{-1}(y)|=\rho,\:y_{n+1}>0}\langle\mathbb U^{(\delta, y_{n+1})}(\widetilde e^0, \widetilde w^0), \mu_{\delta}\rangle y_{n+1}^md\sigma
	\Big|
	\\
	&\leq \int_{C^{-1}\rho\leq |y|\leq C\rho }\left| \text{div}_\delta\left(\mathbb U^{(\delta,y_{n+1})}(\widetilde e^0, \widetilde w^0)y_{n+1}^m\right)\right|dy.
\end{align*}
Then equation \eqref{eq:geom:inv:8} follows from the following:

\vspace{0.2cm}
\noindent
{\bf{Claim 4.}} We have $\text{div}_\delta\left(\mathbb U^{(\delta,y_{n+1})}(\widetilde e^0, \widetilde w^0)y_{n+1}^m\right)=0$.

Indeed, for $i,j=1,2,...,n+1,$
\begin{align}\label{eq:geom:inv:9}
	\left(\mathbb U_i^{(\delta,y_{n+1})}(\widetilde e^0,\widetilde w^0)y_{n+1}^m \right)_{,i}
	=((\widetilde e^0_{ij,j}-\widetilde e^0_{jj,i})y_{n+1}^m)_{,i}+m\left((\widetilde e^0_{i\,n+1}-2\widetilde w^0_{,i})y_{n+1}^{m-1} \right)_{,i}\,.
\end{align}
Here and in what follows we are adopting the Einstein summation convention, meaning that repeated indices are summed.
We will show that each of the two terms on the right-hand side of \eqref{eq:geom:inv:9} vanish.
We calculate
\begin{align*}
	(\widetilde e^0_{ij,j}-\widetilde e^0_{jj,i})
	&=(\widetilde a|y|^{1-n-m}\delta_{ij})_{,j}-(\widetilde a|y|^{1-n-m}\delta_{jj})_{,i}
	\\
	&=-\widetilde an(|y|^{1-n-m})_{,i}
	\\
	&=\widetilde an(n+m-1)|y|^{-1-n-m}y_i.
\end{align*}
Hence,
\begin{align*}
	\left((\widetilde e^0_{ij,j}-\widetilde e^0_{jj,i})y_{n+1}^m\right)_{,i}
	&=\widetilde an(n+m-1)(|y|^{-1-n-m}y_iy_{n+1}^m)_{,i}
	\\
	&=\widetilde an(n+m-1)\Big((-1-n-m)|y|^{-3-n-m}y_iy_iy_{n+1}^m
	\\ 
	&\hspace{1cm}+(n+1)|y|^{-1-n-m}y_{n+1}^m+m|y|^{-1-n-m}y_i\delta_{i\,n+1}y_{n+1}^{m-1}\Big)
	\\
	&=0.
\end{align*}
For the second term in the right-hand side of \eqref{eq:geom:inv:9}, observe that
\begin{align*}
	\widetilde e^0_{i\,n+1}-2\widetilde w^0_{,i}
	&=\widetilde a|y|^{1-n-m}\delta_{i\,n+1}-\widetilde a(|y|^{1-n-m}y_{n+1})_{,i}
	\\
	&=\widetilde a(n+m-1)|y|^{-1-n-m}y_iy_{n+1}.
\end{align*}
Thus,
\begin{align*}
	m\left((\widetilde e^0_{i\,n+1}-2\widetilde w^0_{,i})y_{n+1}^{m-1}\right)_{,i}
	&=m(n+m-1)\widetilde a\Big((|y|^{-1-n-m}y_iy_{n+1})_{,i}y_{n+1}^{m-1}
	\\
	&\hspace{1cm}+|y|^{-1-n-m}y_iy_{n+1} (y_{n+1}^{m-1})_{,i}\Big)
	\\
	&=m(n+m-1)\widetilde a\Big((-1-n-m)|y|^{-3-n-m}y_iy_iy_{n+1}^m
	\\
	&\hspace{1cm}+(n+1)|y|^{-1-n-m}y_{n+1}^m+|y|^{-1-n-m}y_i\delta_{i\,n+1}y_{n+1}^{m-1}
	\\
	&\hspace{2cm}+(m-1)|y|^{-1-n-m}y_i\delta_{i\,n+1}y_{n+1}^{m-1}\Big)
	\\
	&=0.
\end{align*}
This proves Claim 4 and ends the proof for $0<m<1.$

When $-1<m<0$, recall that in this case, $$e_{ij}(x)=a|x|^{1-n-m}\delta_{ij}+O'(|x|^{\epsilon-n-m}),\qquad{i,j=1,...,n+1},$$ with $O'(|x|^{\epsilon-n-m})=O'(|x|^{\epsilon-n-m})\delta_{ij}+x_{n+1}O'(|x|^{\epsilon-n-m-1})$ and the analogous assumptions hold in the $\widetilde{\phi}$-coordinate system. The proof follows the same lines as the case $0<m<1$ with slight modifications by noticing that
$$Z_{\alpha,n+1}=O(x_{n+1}|x|^{-\tau-1}), \:\:\: Z_{n+1,\alpha}=O(x_{n+1}|x|^{-\tau-1}),$$ 
$$(\Phi^*\widetilde{e})_{n+1\:\alpha}=O(x_{n+1}|x|^{-\tau-1}),$$ 
and $$\Phi^*\delta-\delta=L_Z\delta+Q,$$ where $Q=O(|x|^{-2\tau})$ and $Q_{n+1 \: \alpha}=O(x_{n+1}|x|^{-2\tau-1})$ for $\alpha=1,2,...,n$. This ends the proof of Theorem~\ref{thm:geom:inv}.

%%%%%%%%%%%%%%%%%%%%%%%%%%%%%%%%%%%%%%%%%%%%%%%%%%%%%%%%%%%

%%%%%%%%%%%%%%%%%%%%%%%%%%%%%%%%%%%%%%%

\appendix
\section{Appendix}

In this section, we provide some calculations used in Section \ref{sec:interpr:mass}. Let $(M, g, v^mdvol_g)$ be a SMMS. We first calculate the variation $R^{\bullet}_{(g,v,m)}(e,w)$.

\begin{proposition}\label{propo:R:var}
	We have
	\begin{align*}
		R^{\bullet}_{(g,v,m)}(e,0)
		&=\text{div}_g(\text{div}_ge-d\text{tr}_ge)+mv^{-1}(\text{div}_ge-d\text{tr}_ge)(\nabla_gv)
		\\
		&\hspace{0.3cm}+mv^{-1}\left( (m-1)v^{-1}e(\nabla_gv, \nabla_gv)+\text{div}_ge(\nabla_gv)+\langle e,\nabla^2_gv\rangle_g\right)
		\\
		&\hspace{0.6cm}+\langle -\text{Ric}_g+mv^{-1}\nabla^2_gv,e\rangle_g
	\end{align*}
	and
	\begin{align*}
		R^{\bullet}_{(g,v,m)}(0,w)
		&= 2mv^{-1}\left(v^{-1}(\Delta_gv) w+(m-1)v^{-2}|\nabla_gv|^2_gw -(m-1) v^{-2}w\right)
		\\
		&\hspace{0.3cm}-2mv^{-1}\left(\Delta_gw+(m-1)v^{-1}g(\nabla_gv,\nabla_gw) \right).
	\end{align*}
	In particular, 
	\begin{align*}
		v^mR^{\bullet}_{(g,v,m)}(e,0)
		&=\text{div}_g\left( v^m(\text{div}_ge-d\text{tr}_ge) \right)+m\,\text{div}_g(v^{m-1}\nabla_gv\righthalfcup e)
		\\
		&\hspace{0.3cm}+\langle -\text{Ric}_g+mv^{-1}\nabla^2_gv,e\rangle_g v^m
	\end{align*}
	and
	\begin{align*}
		v^mR^{\bullet}_{(g,v,m)}(0,w)
		&=2mv^{m-3}\left( v\Delta_gv+(m-1)|\nabla_g v|_g^2-(m-1)\right)w
		\\
		&\hspace{0.3cm}-2m\,\text{div}_g(v^{m-1} dw).
	\end{align*}
\end{proposition}
\begin{proof}
	We write
	$$
	R_{(g,v,m)}=R_g-mF_{g,v},
	$$
	where 
	$$
	F_{g,v}=2v^{-1}\Delta_gv+(m-1)v^{-2}|\nabla_gv|^2_g-(m-1) v^{-2}.
	$$
	Then
	$$
	R^{\bullet}_{(g,v,m)}=R^{\bullet}_g-mF^{\bullet}_{g,v}.
	$$

	We will make use of the well known formula
	$$
	R^{\bullet}_g=\text{div}_g \text{div}_g e-\Delta_g \text{tr}_ge-\langle \text{Ric}_g, e\rangle_g
	$$
	and directly calculate
	\begin{align*}
		F^{\bullet}_{g,v}&=-2v^{-2}(\Delta_gv) w+2v^{-1}(\Delta_gv)^{\bullet}
		\\
		&\hspace{0.3cm}-2(m-1)v^{-3}|\nabla_gv|_g^2 w+(m-1)v^{-2}(|\nabla_gv|^2_g)^{\bullet}+2(m-1) v^{-3}w.
	\end{align*}
	Since
	$$
	(|\nabla_gv|^2_g)^{\bullet}=-e(\nabla_g v,\nabla_g v)+2g(\nabla_g v, \nabla_g w)
	$$
	and
	$$
	(\Delta_gv)^{\bullet}=\Delta_g w+\frac{1}{2}g(\nabla_g \text{tr}_ge,\nabla_g v)-(\text{div}_g e)(\nabla_g v)-\langle e,\nabla^2_g v\rangle_g,
	$$
	we have
	\begin{align*}
		F^{\bullet}_{g,v}=
		&-2v^{-2}(\Delta_gv) w+v^{-1}g(\nabla_g \text{tr}_ge,\nabla_g v)+2v^{-1}\Delta_g w
		\\
		&-2v^{-1}(\text{div}_ge)(\nabla_g v)-2v^{-1}\langle e, \nabla^2_gv\rangle_g-2(m-1)v^{-3}|\nabla_gv|^2_g w
		\\
		&-(m-1)v^{-2}e(\nabla_g v,\nabla_g v)+2(m-1)v^{-2}g(\nabla_g v, \nabla_g w)+2(m-1) v^{-3}w.
	\end{align*}

	The result now follows.
\end{proof}

We now turn to variations $\mathcal H^{\bullet}_{(g,v,m)}(e,w)$.

\begin{proposition}\label{propo:H:var}
	We have
	\begin{align*}
		2\mathcal H^{\bullet}_{(g,v,m)}(e,w)&=(d\text{tr}_ge-\text{div}_ge)(\eta)-\text{div}_{g|_S}(\eta \righthalfcup e)
		-\langle \pi_g,e\rangle_{g|_S}\\
		&\hspace{0.3cm}+(-2mv^{-1}\nabla_g v\righthalfcup e+mv^{-1}e(\eta,\eta)dv+2mv^{-1}dw)\eta\\
           &\hspace{0.6cm}-2mv^{-2}w dv(\eta)-2mv^{-2}w,		
	\end{align*}
	where $\pi_g$ and $\eta$ are the second fundamental form and the outward pointing unit normal vector on some embedded hypersurface $S \subset \{v>0\}$ respectively. In particular, 
\begin{align*}
		v^m2\mathcal H^{\bullet}_{(g,v,m)}
		&=v^m(d\text{tr}_ge-\text{div}_ge)(\eta)-\text{div}_{g|_S}(v^m \eta\righthalfcup e)
		\\
		&\hspace{0.6cm}-\langle \pi_g,e\rangle_{g|_S}v^m-mv^{m-1}(\nabla_gv\righthalfcup e)(\eta)\\
		&\hspace{0.6cm}+2mv^{m-1}dw (\eta)-2mv^{m-2}w(dv(\eta)+1).
	\end{align*}
\end{proposition}
\begin{proof}
	We write
	$$
	2\mathcal H^{\bullet}_{(g,v,m)}=2 H^{\bullet}_g+2mv^{-1}(dv(\eta))^{\bullet}-2mv^{-2}wdv(\eta)-2mv^{-2}w
	$$
	and make use of the well known formula
	$$
	2H^{\bullet}_g=(d\text{tr}_ge-\text{div}_ge)(\eta)-\text{div}_{g|_S}(\eta \righthalfcup e)-\langle \pi_g,e\rangle_{g|_S}.
	$$
	The first result follows by applying the following:
	$$(dv(\eta))^{\bullet}=(-\nabla_g v\righthalfcup e+\frac{1}{2}e(\eta,\eta)dv+dw)\eta. $$

Then,
\begin{align*}
		v^m2\mathcal H^{\bullet}_{(g,v,m)}
		&=v^m(d\text{tr}_ge-\text{div}_ge)(\eta)-\text{div}_{g|_S}(v^m \eta\righthalfcup e)
		\\
		&\hspace{0.6cm}+mv^{m-1}e(\eta,\nabla^{T}_gv)-\langle \pi_g,e\rangle_{g|_S}v^m\\
		&\hspace{0.6cm}-2mv^{m-1}(\nabla_gv\righthalfcup e)(\eta)+mv^{m-1}e(\eta,\eta)dv (\eta)\\
		&\hspace{0.6cm}+2mv^{m-1}dw (\eta)-2mv^{m-2}w dv (\eta)-2mv^{m-2}w\\
		&=v^m(d\text{tr}_ge-\text{div}_ge)(\eta)-\text{div}_{g|_S}(v^m \eta\righthalfcup e)
		\\
		&\hspace{0.6cm}-\langle \pi_g,e\rangle_{g|_S}v^m-mv^{m-1}(\nabla_gv\righthalfcup e)(\eta)\\
		&\hspace{0.6cm}+2mv^{m-1}dw (\eta)-2mv^{m-2}w(dv(\eta)+1).
	\end{align*}
\end{proof}

%%%%%%%%%%%%%%%%%%%%%%%%%%%%%%%%%%%%%%%

\bigskip\noindent
{\bf{Acknowledgement.}} Part of this work was carried out during the first author's visit to Nanjing University of Science and Technology in the summer of 2025. He is very grateful for the hospitality of the School of Mathematics and Statistics. 
The authors used ChatGPT (OpenAI) to assist in checking mathematical arguments and improving the presentation of this manuscript. The authors take full responsability for all mathematical content.

%\bigskip\noindent
%\textcolor{red}{
%{\bf{On the use of A.I..}} The first arxiv version of this paper was entirely prepared without the help of A.I.. For the current arxiv version, Chat GPT Plus was used to revise and improve the presentation, especially Section 5 and the introduction. Some mathematical discussion with Chat GPT Plus was helpful. The idea inspiring Remark \ref{rmk:inter:H} was suggested by Chat GPT Plus.   
%}

\bigskip\noindent
\textsc{S\'ergio Almaraz\\
Instituto de Matem\'atica e Estat\' istica, \\
Universidade Federal Fluminense\\
Rua Prof. Marcos Waldemar de Freitas S/N,
Niter\'oi, RJ,  24.210-201, Brazil.}\\
e-mail: {\bf{sergioalmaraz@id.uff.br}}

\bigskip\noindent
\textsc{Levi Lopes de Lima\\
Universidade Federal do Cear\'a (UFC),\\
Departamento de Matem\'{a}tica, Campus do Pici, Av. Humberto Monte, s/n, Bloco 914, 60455-760,\\
Fortaleza, CE, Brazil.}\\
e-mail: {\bf{levi@mat.ufc.br}}

\bigskip\noindent
\textsc{Shaodong Wang\\
School of Mathematics and Statistics,\\
Nanjing University of Science and Technology\\
Nanjing, 210094, People’s Republic of China.} \\
e-mail: {\bf{shaodong.wang@mail.mcgill.ca}}

\end{document}